\documentclass{article}
\usepackage{graphicx} 
\usepackage{amsmath,amssymb,amsfonts,amsthm,mathtools}
\usepackage{todonotes, hyperref}
\usepackage{bm, dsfont} 
\usepackage[left=3cm, right=3cm,top=2cm,bottom=2cm]{geometry}

\usepackage{tikz}
\usetikzlibrary{positioning, shapes.geometric, arrows.meta, shadows, calc, backgrounds,cd, fit}
\usepackage{pgfplots}
\pgfplotsset{compat=1.18} 
\usepackage{tikz-3dplot}
\usepackage[normalem]{ulem} 
\usepackage{booktabs} 
\usepackage{afterpage}
\usepackage{wrapfig} 
\usepackage{subcaption} 

\newcommand{\x}{\bm{x}}
\newcommand{\y}{\bm{y}}
\newcommand{\z}{\bm{z}}
\renewcommand{\S}{\mathbb{S}}
\newcommand{\R}{\mathbb{R}}

\newcommand{\B}{\mathbb{B}}
\newcommand{\N}{\mathbb{N}}
\newcommand{\T}{\mathbb{T}}
\newcommand{\D}{\mathcal{D}}
\newcommand{\E}{\mathcal{E}}
\newcommand{\U}{\Upsilon}
\newcommand{\I}{\mathcal I}
\renewcommand{\AA}{\mathcal{A}}

\newcommand{\norm}[1]{\left\lVert \smash{#1} \right\rVert}
\renewcommand{\d}{\, \mathrm{d}} 
\newcommand{\dx}{\mathrm{d}} 

\newcommand{\supp}{\textrm{supp\,}}

\newtheorem{theorem}{Theorem}[section]
\newtheorem{lemma}[theorem]{Lemma}

\theoremstyle{definition}
\newtheorem{definition}[theorem]{Definition}
\newtheorem{example}[theorem]{Example}
\newtheorem{remark}[theorem]{Remark}
\numberwithin{equation}{section}

\newcommand{\changed}[1]{{#1}} 

\title{Sensitivity Analysis on the Sphere and a Spherical ANOVA Decomposition}
\author{ 
	\href{https://orcid.org/0000-0000-0000-0000}{\includegraphics[scale=0.06]{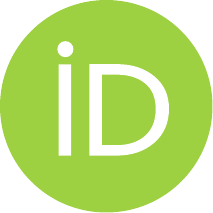}\hspace{1mm}Laura Weidensager} \\
	Department of Mathematics\\
	Simon Fraser University\\
	Canada\\
	\texttt{laura\_weidensager@sfu.ca}\\
}
\date{\today}

\providecommand{\keywords}[1]
{\textbf{\textit{Keywords --}} #1}

\providecommand{\msc}[1]
{ \textbf{\textit{Mathematics Subject Classification (2020) --}} #1}

\begin{document}

\maketitle
\begin{abstract}
     We establish sensitivity analysis on the sphere. We present formulas that allow us to decompose a function $f\colon \mathbb S^d\rightarrow \mathbb R$ into a sum of terms $f_{\boldsymbol u,\boldsymbol \xi}$. The index $\boldsymbol u$ is a subset of $\{1,2,\ldots,d+1\}$, where each term $f_{\boldsymbol u,\boldsymbol \xi}$ depends only on the variables with indices in $\boldsymbol u$. In contrast to the classical analysis of variance (ANOVA) decomposition, we additionally use the decomposition of a function into functions with different parity, which adds the additional parameter $\boldsymbol \xi$. The natural geometry on the sphere naturally leads to the dependencies between the input variables. Using certain orthogonal basis functions for the function approximation, we are able to model high-dimensional functions with low-dimensional variable interactions. 
     \par\medskip 
    \par\medskip
    \msc{41A10, 41A63, 65D15, 65D40}
\end{abstract}

\section{Introduction}
Functions on the sphere arise in diverse applications, ranging from geophysics to quantum mechanics. To efficiently analyze and approximate such functions, a deep understanding of their internal structure is essential. 
The ANOVA (analysis of variance) decomposition provides a remarkably elegant framework for understanding the structure of multivariate functions. In the classical setting, where the domain is a tensor product space, the decomposition exhibits a particularly transparent form: component functions corresponding to different variable subsets are orthogonal, and the hierarchical interaction structure can be characterized cleanly. However, when the domain is the high-dimensional sphere, this clarity gets lost. The lack of a natural product structure complicates both the definition and computation of ANOVA components, making the extension of classical results to spherical domains highly nontrivial. \changed{A general explicit definition of an unique ANOVA decomposition for dependent input variables is not possible, ~\cite{Ho07, LiRa12, Ra142, PoWe24}. We study here an iterative definition for the spherical case, where we are able to transfer the properties of the classical decomposition partially to the sphere. }\par

In this paper, we develop a spherical ANOVA decomposition that allows us to split the total variance of the function into contributions of specific groups of variables. This approach not only yields valuable insights into variable importance but also forms the theoretical foundation for adaptive approximation schemes. \changed{We focus here on an analysis-based approach for describing the ANOVA decomposition on the sphere for the first time. This approach does not answer all application-based questions, which have already been resolved for the classical ANOVA decomposition: For example an approximation theoretical point of view that investigates which function spaces can be described using low-dimensional approximations.}\par

The key difference to the classical decomposition is that we do not have a tensor product structure of the domain here, which makes modeling more difficult and forces the introduction of additional parameters and structures. We found that the decomposition of the function into terms of different parity is a key ingredient for the spherical ANOVA decomposition.\par
\bigskip 
\noindent
\textbf{Our main contributions}
\begin{itemize}
    \item We propose a spherical ANOVA decomposition for the uniform distribution on the sphere $\S^d$. Due to the structure of the sphere, we have to additionally distinguish between even and odd terms and therefore introduce a parity vector.
    \item \changed{We analyze which properties of the ANOVA decomposition can be transferred from the classical case and which problems arise with the dependencies of the variables. }
    \item We introduce reduced orthogonal bases to describe the lower dimensional terms, which we use to implement high-dimensional function approximation on the sphere. Additionally, we define Sobol indices on the sphere in order to apply sensitivity analysis for the spherical ANOVA decomposition.  
    \item We show in numerical examples that the spherical ANOVA decomposition captures the expected non-zero spherical ANOVA terms, similar to how one would expect the terms from the classical ANOVA decomposition. 
\end{itemize}

However, extending the ANOVA decomposition from a tensor product domain to the sphere leads to theoretical challenges arising from the geometry of the domain. The intrinsic curvature of the sphere automatically induces variable interactions. Even if the function has a simple structure in cartesian coordinates, the manifold's constraints enforce coupling between variables, effectively creating artificial interactions. Second, the natural algebraic dependence of the variables (i.e., $\norm{\x} =1$) implies that the inputs cannot vary independently. Consequently, it is generally impossible to construct a decomposition consisting of strictly orthogonal terms in the classical sense.\par

The analysis of functions on the sphere, the ball and the simplex are related. Here, we focus on the analysis of the sphere, but the key idea can be transferred to the unit ball and the simplex. Due to the different geometry on the unit ball and the simplex, the dependencies between the ANOVA terms must be considered carefully. So, we leave the considerations of our approach on the ball and the simplex open for future work. %

\bigskip 
\noindent
\textbf{Related work}\\
Several alternative approaches for dimension reduction and approximation on the sphere exist. Geometric approaches such as those described in \cite{Chu2005} attempt to approximate the data using optimal (rotated) submanifolds. Global polynomial approximation methods for scattered data typically require a number of samples exponential in $d$, as described in \cite{FiTh08}.

Employing the Dirac measure for high-dimensional approximation leads to the anchored ANOVA decomposition, see for example~\cite{KuSlWaWo09}. Due to the geometric properties on the sphere, it is not possible to to generalize this decomposition to the sphere. \par 

In the context of sensitivity analysis, quantifying the influence of groups of variables remains a challenging task when inputs are correlated. As recently discussed by~\cite{La24}, defining the explicit contribution of subsets of variables requires careful consideration in the presence of dependence.
Recent advances in global sensitivity analysis for dependent variables also include kernel-embedding techniques and Shapley effects,~\cite{Da21} as well as random Fourier features,~\cite{PoWe24}.

\changed{To reduce the specification of the ANOVA decomposition to the choice of the coordinate system, in~\cite{BaMeWaSt24} it was studied how to get a low dimensional classical ANOVA decomposition by using an appropriate basis transform. This transformation prevents interpretation using the original variables, so we are not taking that approach here. }

To date, no dedicated ANOVA decomposition for the sphere has been established in the literature. The approach proposed in this work fills this gap and aims to identify the importance of the original physical variables and their specific interactions.\par
\bigskip 
\noindent
\textbf{Structure of this work}\\
\changed{In the next section we introduce the classical ANOVA decomposition for tensor product domains and highlight the differences in the sphere.} We generalize the projection operator from the classical ANOVA decomposition in Section~\ref{sec:projection}. With this, we introduce an ANOVA operator in Section~\ref{sec:ANOVA_operator}. We show properties of this operator, related to the parity, in Section~\ref{sec:parity}. \par
In Section~\ref{sec:ANOVA_sphere}, we finally introduce the spherical ANOVA decomposition, for which we show orthogonality and redundancy relations in Section~\ref{sec:redundancy}. In Section~\ref{sec:challenge}, we highlight the problems that naturally arise with even functions and show how we can fix them. In Section~\ref{sec:basis}, we introduce the orthogonal basis on the lower dimensional balls, which we use to describe the spherical ANOVA terms. How we proceed numerically is described in more detail in Section~\ref{sec:low_dim} and finally in Section~\ref{sec:numerics} we apply our algorithms to some test functions to calculate the spherical ANOVA decomposition of these functions.

\bigskip 
\noindent
\textbf{Notation}\\
We introduce some notation, which we will use in this paper. We use bold face for vectors and matrices. The parameter $d$ is always the ambient dimension. We denote the sphere and ball of radius $r$ by
\begin{align*}
\mathbb S^d_r&=\left\{ \bm x\in \R^{d+1}: \|\bm x\|_2=r\right\},\\
     \mathbb B^{d+1}_r&=\left\{ \bm x\in \R^{d+1}: \|\bm x\|_2\leq r\right\}.
\end{align*}
As usual, we omit the parameter $r$ for the unit sphere $\S^d$ and unit ball $\B^{d+1}$.
Let the surface measure on $\S^d$ be denoted by $\mu_d$ and we assume it is normalized such that 
$$\int_{\S^d} \d \mu_d(\x) = \frac{2\pi^{\frac{d+1}{2}}}{\Gamma\left(\frac{d+1}{2}\right)} \eqqcolon \omega_d.$$
Furthermore, $\S^0 \coloneqq \{-1,1\}$.  Let $[d+1] = \{1,2,\ldots, d+1\}$. For every index $\bm u\subseteq [d+1]$ we denote the corresponding partial vector of the vector $\bm x$ by $\bm x_{\bm u} = (x_i)_{i\in \bm u}$. For every subset $\bm u\subset[d+1]$ we denote the complement by $\bm u^c\coloneqq [d+1]\backslash \bm u$. We compose a vector from 2 (or more) parts by writing $\bm c = (\bm a_{\bm u},\bm b_{\bm u^c} )$, which means for the vector $\bm c$ that $\bm c_{\bm u} = \bm a_{\bm u}$ and $\bm c_{\bm u^c} = \bm b_{\bm u^c}$. We use the usual notation for $L_2$-spaces and the corresponding scalar products on the domain $\mathbb D \in \{\B^{d+1},\S^{d}\}$,
\begin{align*}
    \norm{f}_{L_2(\mathbb D)} &= \int_{\mathbb D}|f(\x)|^2\dx \x, \quad 
    L_2(\mathbb D) = \{f\colon \mathbb D\rightarrow \R \mid \norm{f}_{L_2(\mathbb D)}<\infty \},\\
    \langle f,g\rangle_{\mathbb D}&=\int_{\mathbb D}f(\x) g(\x)\dx \x.
\end{align*}

\section{\changed{Classical versus spherical ANOVA decomposition}}
The classical ANOVA (analysis of variance) decomposition was introduced by Hoeffding~\cite{Ho48} and has been studied in many contexts, see for example~\cite{EfSt81, JiOw03, Wa90,NW08,KuSlWaWo09} and references therein. Often this model is also called high-dimensional model representation, see also~\cite{RaAl99} for an overview. 
The ANOVA decomposition decomposes a $d$-variate function into $2^d$ ANOVA terms where each term belongs to a subset $\bm u \subseteq [d]$ 
of coordinate indices. Each single term only depends on the variables in the corresponding subset and the number of
these variables is the order of the ANOVA term. The ANOVA decomposition is a valuable tool for understanding the structure and behavior of functions and enables to focus on the most significant factors when constructing approximation models.
It has extensively been used for the analysis of quasi-Monte Carlo methods for integration, see e.g.\ \cite{CaMoOw97, LeOw02, LiOw06, So01, GrHo10} and references therein. Applications to sparse grids can be found in~\cite{BuGr04, Gr06, He03, Holtz11}.\par

Let $\mathbb D^d\in \{\T^d,\R^d,[0,1]^d\}$ be a tensor product domain, for a function $f\colon\mathbb D^d\rightarrow \R$  the \textbf{ANOVA decomposition} is given by 
\begin{equation}\label{eq:anova-decomp}
f(\bm x)=f_{\varnothing}+\sum_{i=1}^d f_{\{i\}}(x_i)+\sum_{i\neq j=1}^d f_{\{i,j\}}(x_i,x_j)+\cdots+f_{[d]}(\bm x)=\sum_{\bm u\subseteq  [d]} f_{\bm u}(\bm x_{\bm u}).
\end{equation}
It is common to refer to the terms $f_{\bm u}$ where  $|\bm u| = q$ collectively as the \textbf{order-$q$ terms}. In general, there are many possibilities for a decomposition~\eqref{eq:anova-decomp}. Indeed, an arbitrary choice of $f_{\bm u}$ for all $|\bm u| < d$ can be accommodated by taking $f_{[d]}$ to be $f$
minus all the other terms. But with some additional conditions the decomposition is unique. 
Typically, one demands orthogonality of the involved terms with respect to \changed{a tensor product density $\mu(\x) = \prod_{i=1}^d \mu_i(x_i)$}. 
Examples for the application of the classical ANOVA decomposition can be found in:

\begin{itemize}
    \item For periodic functions on the torus $\T^d$, trigonometrical polynomials are used as basis in~\cite{PoSc19a}, wavelet functions in~\cite{LiPoUl23} \changed{and Gaussian mixtures in~\cite{HeBaSt22}.}
    \item On the cube $[0,1]$ the half-period cosine basis was studied in~\cite{PoSc19b}.
    \item For the use of random Fourier features for the generalized ANOVA decomposition on $\R^d$ see~\cite{PoWe24}.
\end{itemize}
The classical ANOVA decomposition has the following properties
\begin{itemize}
    \item For the definition of the terms a projection operator is used, which integrates the function over all variables not in $\bm u$ by
    \begin{equation}\tag{\changed{I}} \label{eq:P_u_classic}
       \changed{ P_{\bm u } f = \int_{\mathbb D^{|d\backslash \bm u|}} f(\x) \d\left( \prod_{i\in [d]\backslash \bm u}\mu_{i}(x_i) \right)= \left[\prod_{i\in \bm u} P_i\right] f.}
    \end{equation}
    \item The ANOVA terms are defined iteratively by
    \begin{align}\tag{\changed{II}}\label{eq:moebius_classical}
    \changed{f_{\bm u }(\x_{\bm u}) }&= \changed{[P_{\bm u}f](\x_{\bm u})-\sum_{\bm v \subset \bm u} f_{\bm v}(\x_{\bm v}), \quad f_{\varnothing} = P_{\varnothing}f },\\
    \intertext{or equivalently,}
    f_{\bm u }(\x_{\bm u})& = \sum_{\bm v \subseteq \bm u} (-1)^{|\bm u|-|\bm v|} P_{\bm v}f.\notag
    \end{align}
        \item The ANOVA terms fulfill the integral conditions
    \begin{equation}\tag{\changed{III}}\label{eq:integal_cond_classic}
    \int_{\mathbb D}f_{\bm u }(\x_{\bm u}) \d x_i = 0 \quad \text{ for all } i\in \bm u.
    \end{equation}
    \changed{This condition is necessary for the uniqueness of the decomposition.}
    \item All terms are orthogonal,
    \begin{equation}\label{eq:orth_classic}\tag{\changed{IV}}
    0=\langle f_{\bm u},f_{\bm v}\rangle_{\mu} \quad \text{ for all } \quad \bm u\neq \bm v.
    \end{equation} 
    This is a necessary condition for sensitivity analysis, where the total variance of the function is the sum of the variances of the ANOVA terms.  
\end{itemize}
\begin{figure}[tb]
\begin{tikzpicture}
\centering
\pgfplotsset{
    my view/.style={
        view={315}{25}, 
        width=11cm, height=11cm,
        xmin=-1.5, xmax=1.5,
        ymin=-1.5, ymax=1.5,
        zmin=-1.5, zmax=1.5,
        axis lines=center,
        xlabel style={anchor=north west},
        ylabel style={anchor=south west},
        zlabel style={anchor=south},
        xtick={-1,1}, ytick={-1,1}, ztick={-1,1},
        tick label style={font=\tiny},
        colormap/viridis,
    }
}
\begin{axis}[
    name=cube_axis,
    my view,
    title={\textbf{Cube $[-1,1]^3$}},
    colorbar,
    colorbar style={
    at={(-0,0.9)},          
    height=0.6*\pgfkeysvalueof{/pgfplots/parent axis height}, 
    ytick={0, 0.5, 1},     
}
]

\addplot3[surf, opacity=1, shader=interp, domain=-1:1, y domain=-1:1, point meta={z^2}] ({x}, {-1}, {y});
\addplot3[surf, opacity=1, shader=interp, domain=-1:1, y domain=-1:1, point meta={1}] ({x}, {y}, {1});
\addplot3[surf, opacity=1, shader=interp, domain=-1:1, y domain=-1:1, point meta={z^2}] ({-1}, {x}, {y});

\draw[ ->, black] (axis cs:-2.3, 0, 0) -- (axis cs:2.5, 0, -0);
\draw[ ->,black ] (axis cs: 0, -2.3,0) -- ( axis cs:0, 2.3,0);
\draw[->, black] (axis cs: 0, 0,-1.8) -- ( axis cs:0, 0,1.8);
\draw[ ->, red,thick] (axis cs:-1, -1, -0.5) -- (axis cs:1, -1, -0.5);
\draw[ ->, red,thick] (axis cs:-1, -1, 0.4) -- (axis cs:1, -1, 0.4);
\draw[ ->, red,thick] (axis cs:-1, -1, -0.5) -- (axis cs:-1, 1, -0.5);
\draw[ ->, red,thick] (axis cs:-1, -1, 0.4) -- (axis cs:-1,1, 0.4);

\node[anchor=north west] at (axis cs: 2, 0, 0) {$x_2$};
\node[anchor=south west] at (axis cs: 0, 2, 0) {$x_3$};
\node[anchor=south]      at (axis cs: 0, 0, 1.8) {$x_1$};
\node[anchor=south]      at (axis cs: 2, 2, 5) {$P_1$};

\end{axis}
\begin{axis}[
    at={(cube_axis.outer east)}, anchor=outer west,
    xshift=1.5cm,
    my view,
    title={\textbf{Sphere $\mathbb{S}^2\subset \R^3$}},
]

\addplot3[
    surf,
    opacity=0.3,
    z buffer=sort,
    domain=0:360,
    y domain=-90:90,
    samples=35, samples y=15,
    point meta={(z)^2} 
] ({cos(x)*cos(y)}, {sin(x)*cos(y)}, {sin(y)});
\addplot3 [
    thick, red, ->,
    domain=0:360,
    samples=40
] ({0.98*cos(x)}, {sin(x)*0.98}, {0.2}) ;
\addplot3 [
    thick, red, ->,
    domain=0:360,
    samples=40
] ({0.866*cos(x)}, {sin(x)*0.866}, {0.5}) ;
\node[anchor=north west] at (axis cs: 1.6, 0, 0) {$x_2$};
\node[anchor=south west] at (axis cs: 0, 1.7, 0) {$x_3$};
\node[anchor=south]      at (axis cs: 0, 0, 1.5) {$x_1$};
\end{axis}
\end{tikzpicture}
\caption{\changed{Illustration of the function $f(x_1,x_2,x_3) =x_1^2$ to illustrate the differences between a tensor product domain and the sphere $\S^2$. We illustrate for two fixed $x_1$ the integral for the projection $P_1$ in red. On the cube, this is an integral over the two-dimensional cube. On the sphere, the projection $P_1$ is an integral on a circle.}}
\label{fig:cube_sphere}
\end{figure}

\changed{It is not straight forward to construct an ANOVA decomposition on the sphere. So far, there is no other construction available in the literature.}\\

\changed{In Figure~\ref{fig:cube_sphere} we illustrate the main challenges on the sphere. We plot the function $f(x_1,x_2,x_3) = x_1^2$ on the cube and on the sphere. The integrals for the projections $P_1$ are illustrated in red in Figure~\ref{fig:cube_sphere}. On the cube, $P_1$ is an integral over the two-dimensional cube. On the sphere, the projections $P_1$ are integrals over circles, which depend on the size of the circle and therefore on the variable $x_i$. For that reason we will divide by the length of the circle, denoted by $m_1(x_1)$ in~\eqref{eq:Pu_def} to construct a projection operator on the sphere. On the cube $[-1,1]^3$, the classical ANOVA decomposition with respect to the uniform measure is $f_{\varnothing } = \tfrac 23$, $f_{\{1 \}} =x_1^2 - \tfrac 23$ and all other terms are zero. Our aim is to construct an ANOVA decomposition on the sphere, where in this case also only these two ANOVA terms are nonzero. In Example~\ref{ex:x1} we calculate the decomposition on the sphere for this function.}\\

The aim of this paper is to define a spherical ANOVA decomposition for functions $f\colon \S^d \rightarrow \R$ on the sphere, whereby we want to retain as many properties of the classical decomposition as possible. \changed{In Table~\ref{tab:cla_vs_sphere} we compare the spherical to the classical ANOVA decomposition and their properties.}\\

\begin{table}[htbp]
    \centering
    \begin{tabular}{p{0.19\textwidth}|p{0.37\textwidth}|p{0.37\textwidth}}
         & \textbf{classical ANOVA} & \textbf{spherical ANOVA} \\
         \hline
        ANOVA terms & \vspace*{10pt}\[\text{\eqref{eq:anova-decomp}:} \quad f(\bm x)=\sum_{\bm u\subseteq  [d]} f_{\bm u}(\bm x_{\bm u})\] &
\[ \text{\eqref{eq:anova_xi_sphere}:}\quad f(\bm x)
 = \sum_{\bm u\in \D} \sum_{\bm \xi\in \U_{\bm u}}f_{\bm u,\bm \xi}(\x_{\bm u})\]  additional parity vector $\bm \xi$\\
        \hline
        projection $P_{\bm u}$ &\[\text{\eqref{eq:P_u_classic}:}\quad
        P_{\bm u } f= \left[\prod_{i\in \bm u} P_i\right] f \]
        \[P_{\bm u} P_{\bm v} = P_{\bm u\cap \bm v} \quad \text{ for all } \bm u, \bm v \subseteq [d]
        \]
        integration over tensor product & \[
            \text{\eqref{eq:Pu_def}:}\quad P_{\bm u} f \neq 
\left[\prod_{i\in \bm u} P_i\right] f \]
\[
\text{Lem.~\ref{lem:PvPu}: } P_{\bm v} P_{\bm u} = P_{\bm v} \quad \text{ only for }  \bm v \subseteq \bm u
\]
integration over spheres \\
        \hline
        iterative definition & \vspace*{10pt}\[\text{\eqref{eq:moebius_classical}:}\quad 
    f_{\bm u } = P_{\bm u}f -\sum_{\bm v\subset u} f_{\bm v}\] & 
    Definition~\ref{def:ANOVA_operator}: all even parts in highest order term\newline
    Definition~\ref{def:spherical_anova_final}: includes unnecessary terms\newline
    $\rightarrow$ final Definition~\ref{def:final_def} \\
        \hline
        integral conditions  & \[\text{\eqref{eq:integal_cond_classic}:}\,
    \int_{\mathbb D}f_{\bm u }(\x_{\bm u}) \d x_i = 0 \, \text{ for all } i\in \bm u\] & \vspace*{10pt} more complicated integrals in Lem.~\ref{lem:integral_cond} \\
        \hline
        orthogonality & \[\text{\eqref{eq:orth_classic}:} \quad 
    0=\langle f_{\bm u},f_{\bm v}\rangle_{\mu} \,\text{ for all } \, \bm u\neq \bm v.\] 
    orthogonality between all terms &\[\text{\eqref{eq:orth_xi}:}\quad
\langle f_{\bm u,\bm \xi} ,f_{\tilde{\bm u},\tilde{\bm \xi}}\rangle_{\S^d} = 0 \quad \text{ if } \bm \xi\neq \tilde{\bm \xi}\]
orthogonality only for different parity\\
    \end{tabular}
        \caption{\changed{Properties of the ANOVA decomposition.}}
            \label{tab:cla_vs_sphere}
\end{table}

For \changed{other} generalizations of the classical decomposition to some non-tensor product densities see for example~\cite{Ho07, LiRa12, Ra142, PoWe24}. The needed assumption there is that the domain is grid-closed to maintain at least some tensor product structures. But this assumption is not fulfilled on the sphere. 
For the following reasons we can not apply the classical or generalized ANOVA decomposition directly to the sphere:
 \begin{itemize}
     \item Using the classical decomposition on $\R^{d+1}$ or $[0,1]^{d+1}$ is not even with the generalized ANOVA decomposition possible, since $\S^d$ is not a grid-closed domain in $\R^{d+1}$. Problems arise at the intersection points of the axes with the sphere. 
     \item Using spherical coordinates would lead to the loss of interpretability of the variables, which is a key advantage of the ANOVA decomposition.
     \item For tensor product domains, the domain of the lower dimensional functions $f_{\bm u}$ are the lower-dimensional tensor product domains. For functions on the sphere this is not possible. If $f_{\bm u}\colon \S^{|\bm u|-1}\rightarrow \R$, these functions have to be extended to higher dimensional spheres. We will later define the spherical terms $f_{\bm u}\colon \B^{|\bm u|}\rightarrow \R$, defined on lower dimensional balls, such that they are automatically expanded to the sphere $\S^d$. 
     \item For the classical ANOVA decomposition, one uses an orthonormal basis in $L_2$ and then finds a connection between the basis functions and the ANOVA terms. This works well because of the tensor product structure of the basis in the $d$-dimensional function space $L_2$. In contrast, on the sphere $\S^d$, there is no basis of tensor product form, which can be used for this purpose. 
     \item Defining the projection operator by $P_j(f) = f(x_1,\ldots, x_{j-1},c_j,x_{j+1},\ldots, x_d )$ for some anchor point $\bm c = (c_j)_{j=1}^d$ leads to the anchored ANOVA decomposition, see for example~\cite{KuSlWaWo09}. Due to the geometric properties on the sphere, it is not possible to generalize this anchored decomposition to the sphere. On the unit ball $\B^{d+1}$, the anchored ANOVA decomposition with anchor point $\bm 0$ can be calculated  equally to the tensor product domain. 
     
 \end{itemize}
While proposing a spherical ANOVA decomposition in this paper, we will also highlight the differences and challenges in comparison to the classical ANOVA decomposition.

\section{A spherical ANOVA decomposition and its properties}
Analyzing high-dimensional functions on the sphere poses a significant challenge, as separating local dependencies from global structures is often non-trivial. The ANOVA decomposition provides a systematic framework for this task by hierarchically decomposing the function into main effects and higher-order interactions. This enables us not only to quantify the influence of individual variables but also to understand the underlying coupling structure of the function.\par
Denote $\D \coloneqq \{\bm u \subseteq[d+1] \mid |\bm u|\neq d \}$. For a function $f\colon \S^{d}\rightarrow \R$ we will study a decomposition of the form
\begin{equation}\label{eq:anova_sphere}
f(\bm x) = f_{\varnothing} + \sum_{q=1}^{d-1} \sum_{|\bm u|=q} f_{\bm u}\left(\x_{\bm u}\right) +f_{d+1}(\x) = \sum_{\bm u\in \D} f_{\bm u}(\x_{\bm u}),
\end{equation}
where the component functions $f_{\bm u}\colon \B^{|\bm u|} \rightarrow \R$ are functions on lower-dimensional balls and $f_{d+1}\colon \S^{d}\rightarrow \R$. The domain of the functions $f_{\bm u}$ is the lower-dimensional ball $\B^{|\bm u|}$, but naturally with this decomposition, the terms $f_{\bm u}$ are also defined on $\S^d$.
We do not need terms for the indices $|\bm u|=d$, since due to the normalization $\norm{\x}=1$, the last variable would be specified, except for the sign, from fixing all other variables. We treat all these parts of the function $f$ in the very last term $f_{[d+1]}$. \par

We propose to combine the spherical ANOVA decomposition~\eqref{eq:anova_sphere} with the decomposition into even and odd parts. \changed{In the one-dimensional case, a function $f$ is defined as \textbf{odd} if $f(-x)=-f(x)$ and as \textbf{even} if $f(-x)=f(x)$ almost everywhere. Then, an arbitrary function $f$ is decomposed into its even and odd parts by
\[
f(x) = f_{\mathrm e}(x) + f_{\mathrm o}(x), \qquad
f_{\mathrm e}(x) = \tfrac12(f(x)+f(-x)), \quad
f_{\mathrm o}(x) = \tfrac12(f(x)-f(-x)).
\]
The straightforward multivariate generalization is defined in detail in Section~\ref{sec:even_odd}. Here, $\xi_i=1$ belongs to an odd function in direction $x_i$ and $\xi_i=0$ to an even function in direction $x_i$.} We will always use the index $\bm u$ for the spherical ANOVA terms and the index $\bm\xi$ for the parity of the term.

Since the function $f_{\bm u}$ does not depend on the variables $x_i$ for $i\in \bm u^c$, the term $f_{\bm u,\bm \xi}$ \changed{can only be defined if for the parity vector} $\bm \xi_{\bm u^c} = \bm 0$. For that reason we introduce the index-set 
$$\U_{\bm u} \coloneqq \{\bm \xi \in \{0,1\}^{d+1} \mid \bm \xi_{\bm u^c} = \bm 0 \},$$
where $|\U_{\bm u}| = 2^{|\bm u|}$.
We study in this paper the decomposition of the form
\begin{equation}\label{eq:anova_xi_sphere}
f(\bm x)
= \sum_{\bm u\in \D} \sum_{\bm \xi\in \U_{\bm u}}f_{\bm u,\bm \xi}(\x_{\bm u}).
\end{equation}
In this decomposition we have in total
\begin{equation}\label{eq:number_spherical_terms}
\sum_{q=0}^{d-1} \binom{d+1}{q}2^q +2^{d+1} =\sum_{q=0}^{d+1} \binom{d+1}{q}2^q -2^d\,(d+1)= 3^{d+1} -2^d\,(d+1)
\end{equation}
spherical ANOVA terms.
\begin{example}\label{ex:d=2}
For the case $d=2$, we have to include the $15 = 3^3-2^2\cdot 3$ terms
\begin{align*}
&f_{\varnothing,\bm 0},&\text{order 0}\\
&f_{\{1\},(1,0,0)^\top}, f_{\{1\},(0,0,0)^\top}, f_{\{2\},(0,1,0)^\top}, f_{\{2\},(0,0,0)^\top},f_{\{3\},(0,0,1)^\top}, f_{\{3\},(0,0,0)^\top},&\text{order 1}\\
&f_{\{1,2,3\},(0,0,0)^\top}, f_{\{1,2,3\},(0,0,1)^\top}, f_{\{1,2,3\},(0,1,0)^\top}, f_{\{1,2,3\},(0,1,1)^\top},&\text{order 3}\\
& f_{\{1,2,3\},(1,0,0)^\top}, f_{\{1,2,3\},(1,0,1)^\top},f_{\{1,2,3\},(1,1,0)^\top}, f_{\{1,2,3\},(1,1,1)^\top} &\text{order 3}
\end{align*}
Note that we do not need the two-dimensional terms $f_{\bm u}$ for $|\bm u|=2$. 

\end{example}

In the following subsection we introduce an operator $P_{\bm u}$, which we will use in Section~\ref{sec:ANOVA_sphere} to define the spherical ANOVA terms $f_{\bm u,\bm \xi}$.
\subsection{The projection operator}\label{sec:projection}
For an index $\bm u\in \D\backslash [d+1]$ and $\y\in \B^{|\bm u|}$ denote the \textit{fiber}
\begin{equation}\label{eq:fiber_def}
\{\x\in \S^d \subset \R^{d+1} \colon \x_{\bm u} = \y\}
\end{equation}
and the corresponding integrals $m_{\bm u}\colon \B^{|\bm u|}\rightarrow \R$
\begin{equation*}
m_{\bm u}(\y_{\bm u }) \coloneqq \int_{\S^{|\bm u^c|-1}_{\sqrt{1-\norm{\bm y_{\bm u}}^2}}} 
1 \d \mu_{d-|\bm u|}(\x_{\bm u^c}),
\end{equation*}
where $\mu_{d-|\bm u|}(\x_{\bm u^c})$ is the surface measure of the $d-|\bm u|$-dimensional sphere of radius $\sqrt{1-\norm{\y_{\bm u }}^2} = \norm{\x_{\bm u^c}}$.\par

We study the uniform distribution $\mu_d$ on the sphere. Due to~\eqref{eq:int_sphere_r}, the corresponding conditional distribution on the lower dimensional  balls $\B^{|\bm u|}$, is the distribution
\begin{align*}
m_{\bm v}(\y_{\bm v})&  
=\begin{cases} \displaystyle \changed{\left(1-\norm{\y_{\bm v}}^2\right)^{(d-|\bm v|)/2}} \int_{\S^{d-|\bm v|}} 1 \d \mu_{d-|\bm v|}(\x_{}) &\text{ if }|\bm v|\leq d-1\\
0 &\text{ if }|\bm v|= d, \\
1 &\text{ if }\bm v= [d+1], \\
\end{cases}\\
&=\begin{cases}\omega_{d-|\bm v|} \changed{\left(1-\norm{\y_{\bm v}}^2\right)^{(d-|\bm v|)/2}}   &\text{ if }|\bm v|\leq d-1,\\
0 &\text{ if }|\bm v|= d, \\
1 &\text{ if }\bm v= [d+1], \\
\omega_{d} &\text{ if }\bm v = \varnothing.
\end{cases}
\end{align*}

\begin{definition}[Projection operator]
We define for $\bm u \in \D \backslash [d+1]$ the projection operator $P_{\bm u} \colon L_2(\S^d) \rightarrow L_2(\B^{|\bm u|})$ by
\begin{align}\label{eq:Pu_def}
[P_{\bm u} f] (\y_{\bm u}) 
&\coloneqq \frac{1}{m_{\bm u}(\y_{\bm u })}\int_{ \S_{\sqrt{1-\norm{\y_{\bm u }}^2}}^{|\bm u^c|-1 } } f( \x ) \d \mu_{|\bm u^c|-1}(\x_{\bm u^c})\\
&= \frac{1}{\omega_{d-|\bm u|}}\int_{ \S^{|\bm u^c|-1 } } f\left( \y_{\bm u} , \sqrt{1-\norm{\y_{\bm u }}^2} \x_{\bm u^c} \right) \d \mu_{|\bm u^c|-1}(\x_{\bm u^c}).\notag
\end{align}
\end{definition}
The operators $P_{\bm u}$ calculate the mean of the function $f$ on the fiber~\eqref{eq:fiber_def}. 
Note that for $|\bm u| =  d $ this operator is not defined and for $\bm u = [d+1]$ the operator $P_{[d+1]}$ is the identity operator. For one-dimensional $\bm u=\{i\}$, the operator $P_{\bm u}$ is
\begin{align*}
[P_{\{i\}} f] (y_i) 
&= \changed{\omega^{-1}_{d-1} \left(1-y_i^2\right)^{-(d-1)/2}}\int_{ \S_{\sqrt{1-y_i^2}}^{d-1 } } f( \x ) \d \mu_{d-1}(\x_{[d+1]\backslash i})\\
&=  \frac{1}{\omega_{d-1}}\int_{ \S^{d-1 } } f\left( y_i , \sqrt{1-y_i^2}\, \x_{\bm u^c} \right) \d \mu_{d-1}(\x_{i^c}).
\end{align*}
For $d=2$, the operator $P_{\{1\}}$ is illustrated in Figure~\ref{fig:P_1}.  \par
\tdplotsetmaincoords{70}{10}
\begin{figure}[tb]
\begin{center}
\begin{tikzpicture}[scale=2.5,tdplot_main_coords]
x
    \def\R{1.0}        
    \def\xval{0.6}     
    \def\cradius{sqrt(\R*\R - \xval*\xval)} 

    \draw[->] (-1.2,0,0) -- (1.5,0,0) node[anchor=north east]{$x_1$};
    \draw[->] (0,-1.5,0) -- (0,3,0) node[anchor=north west]{$x_2$};
    \draw[->] (0,0,-1.2) -- (0,0,1.5) node[anchor=south]{$x_3$};

    \shade[ball color=gray!20, opacity=0.4] (0,0) circle (1cm);
    \draw (0,0) circle (1cm);
    
    \draw[thick, red, domain=0:360, samples=100] 
        plot ({\xval}, {\cradius*cos(\x)}, {\cradius*sin(\x)});
    
    \fill[red, opacity=0.1] 
        plot[domain=0:360, samples=50] ({\xval}, {\cradius*cos(\x)}, {\cradius*sin(\x)});

    \coordinate (Proj) at (\xval, 0, 0);

    \fill[black] (Proj) circle (0.5pt);
    \node[anchor=north, yshift=-2pt] at (Proj) {\footnotesize $y_1$};

    \node[red, anchor=south west] at (\xval, {\cradius}, 0) 
        {\footnotesize $\mathbb{S}_{\sqrt{1-y_1^2}}^{1}$};
;

\end{tikzpicture}
\end{center}
\caption{Visualization of $P_{\{1\}}f$ for $d=2$: Integration over the red circle with fixed $x_1 = y_1$.}
\label{fig:P_1}
\end{figure}%
Here we see the difference to the classical ANOVA decomposition: In the tensor product setting the projection operator has also tensor product structure, i.e., $P_{\bm u}$ is a product of one-dimensional $P_i$, see~\eqref{eq:P_u_classic}. For the operator~\eqref{eq:Pu_def} on the sphere, we do not have such a structure, but nevertheless the operator $P_{\bm u}$ has some properties, which we will use to define a spherical ANOVA decomposition. First, the sequential application of the projection operator is the same as the operator applied with the smaller index set, as the following lemma shows.  
\begin{figure}[tb]
\centering
\begin{tikzcd}[
    row sep=large, 
    column sep=large,
    execute at begin picture={
        \pgfdeclarelayer{background}
        \pgfsetlayers{background,main}
    },
    execute at end picture={
        \begin{pgfonlayer}{background}
            \node[fit=(top1)(mid1), draw, rounded corners, inner sep=10pt, fill=gray!20] {};
            \node[fit=(top2)(bot2), draw, rounded corners, inner sep=10pt, fill=gray!20] {};
            \node[fit=(top3)(mid3), draw, rounded corners, inner sep=10pt, fill=gray!20] {};
        \end{pgfonlayer}
    }
]
\text{function }&|[alias=top1]| f  & |[alias=top2]| P_{\bm u}f & |[alias=top3]| P_{\bm v}f \\
\text{function space }&|[alias=mid1]| L_2(\S^d) 
    \arrow[r, "P_{\bm u}"] 
    \arrow[rr, bend right=90,looseness=1.5, "P_{\bm v}"'] 
    & 
|[alias=mid2]| L_2(\B^{|\bm u|}) 
    \arrow[d, rightarrow, left, "\tiny{\bm x_{\bm u} \rightarrow \bm x}"] 
    & 
|[alias=mid3]| L_2(\B^{|\bm v|}) \\
&& 
|[alias=bot2]| L_2(\S^d) 
    \arrow[ru, "P_{\bm v}"'] 
    & 
|[alias=bot3]| \phantom{X}
\end{tikzcd}
\caption{\changed{The functions and function spaces involved in Lemma~\ref{lem:PvPu}.}}
\label{fig:Lemma33}
\end{figure}%
\begin{lemma}\label{lem:PvPu}
The operators $P_{\bm u}$ \changed{and $P_{\bm v}$} fulfill for $\bm v\changed{\subseteq} \bm u$,
\begin{equation*}
P_{\bm v} P_{\bm u}f = P_{\bm v} f.
\end{equation*}
\changed{The connection of the involved functions, operators and spaces is illustrated in Figure~\ref{fig:Lemma33}.}
\end{lemma}%
\begin{proof}
Due to the definition,
\begin{align*}
\left[P_{\bm v} P_{\bm u}f \right] (\y_{\bm u}) 
&=\frac{1}{m_{\bm v}(\y_{\bm v })}  \int_{\S_{\sqrt{1-\norm{\y_{\bm v }}^2}}^{|\bm v^c|-1 }} \frac{1}{m_{\bm u}(\y_{\bm u })}\int_{\S_{\sqrt{1-\norm{\y_{\bm u }}^2}}^{|\bm u^c|-1 }} 
f(\x) \d \mu_{d-|\bm u|}(\x_{\bm u^c}) \d \mu_{d-|\bm v|}(\x_{\bm v^c})\\
 &=
\frac{1}{m_{\bm u}(\y_{\bm u })}\int_{\S_{\sqrt{1-\norm{\y_{\bm u }}^2}}^{|\bm u^c|-1 }} \frac{1}{m_{\bm v}(\y_{\bm v })}  \int_{\S_{\sqrt{1-\norm{\y_{\bm v }}^2}}^{|\bm v^c|-1 }} 
f(\x)\d \mu_{d-|\bm v|}(\x_{\bm v^c}) \d \mu_{d-|\bm u|}(\x_{\bm u^c}) \\
 &=
 \frac{1}{m_{\bm v}(\y_{\bm v })}  \int_{\S_{\sqrt{1-\norm{\y_{\bm v }}^2}}^{|\bm v^c|-1 }} 
f(\x)\d \mu_{d-|\bm v|}(\x_{\bm v^c}) = \left[P_{\bm v} f \right] (\y_{\bm v}) .
\end{align*}
This finishes the proof.
\end{proof}
\changed{Note that in the previous Lemma the operators $P_{\bm u}$ and $P_{\bm v}$ have $L_2(\S^d)$ as domain. That means the natural mapping from $\S^d$ to $\B^{|\bm u|}$, $\bm x\rightarrow \x_{\bm u}$ is necessary to apply the operator $P_{\bm v}$ after $P_{\bm u}$.}\par
Applying the operator $P_{\bm a}$ to a function, which depends only on the variables in $\bm u$ leads to the following lemma. This will be helpful in Section~\ref{sec:ANOVA_operator} to define integral conditions for the ANOVA operator. 
\begin{lemma}\label{lem:Pafu=}
\changed{For $\bm u\in \D\backslash \varnothing$ let $g_{\bm u} \colon \B^{|\bm u|}\rightarrow \R$. Then, for all $\bm a\in \D$ with $\bm a \subset \bm u$, the operator $P_{\bm a}$ applied to the function $g_{\bm u}$ can be written as}
    $$ \left[P_{\bm a} g_{\bm u}\right] (\y_{\bm a})=
         \frac{\omega_{d-|\bm u|} }{\omega_{d-|\bm a|} }\int_{\B^{|\bm u\backslash \bm a|}}  \changed{\left( 1-\norm{\x_{\bm u\backslash \bm a}}^2\right)^{(d-|\bm u|-1)/2}}g_{\bm u}\left(\y_{\bm a}, \sqrt{1-\norm{\y_{\bm a}}^2}\,\x_{\bm u\backslash \bm a} \right)\d  \x_{\bm u\backslash \bm a}. $$
For arbitrary $\bm v\in \D$, denote $\bm a = \bm u \cap \bm v$. Then
$$  [P_{\bm v}g_{\bm u}](\y_{\bm v}) = \frac{\omega_{|(\bm u\cup \bm v)^c| -1}}{\omega_{d-|\bm v|}}   \int_{\B^{|\bm u \backslash \bm a|}}\changed{\left(1-\norm{\x_{\bm u \backslash \bm a}}^2\right)^{|(\bm u\cup \bm v)^c|/2 -1}} g_{\bm u}\left(\y_{\bm a}, \sqrt{1-\norm{\y_{\bm v}}^2}\,\x_{\bm u\backslash \bm a}\right)
 \d \x_{\bm u \backslash \bm a}.$$
\end{lemma}
\begin{proof}
We start with $\bm a\subset \bm u$, and use the definition of the operator $P_{\bm a}$ applied to the function $g_{\bm u}$. \changed{For shortening notation denote $r=\sqrt{1-\frac{\norm{\x_{\bm u \backslash \bm a}}^2}{1-\norm{\y_{\bm a}}^2}} =\sqrt{ \frac{1-\norm{\x_{\bm u}}^2}{1-\norm{\y_{\bm a}}^2}}$, then}
\begin{allowdisplaybreaks}
\begin{small}
\begin{align*}
&[P_{\bm a} g_{\bm u}] (\y_{\bm a}) 
= \frac{1}{m_{\bm a}(\y_{\bm a })}\int_{\S^{d-|\bm a|}_{\sqrt{1-\norm{\y_{\bm a}}}}} 
g_{\bm u}(\x_{\bm u}) \d \mu_{d-|\bm a|}(\x_{\bm a^c})\\
&\stackrel{\eqref{eq:split_integral_r}}{=} \frac{1}{m_{\bm a}(\y_{\bm a })} \changed{\left(1-\norm{\y_{\bm a}}^2\right)^{(d-|\bm a|)/2}}\int_{\B^{|\bm u\backslash \bm a|}} \frac{1}{r}g_{\bm u}\left(\x_{\bm u}\right)\int_{\S^{|\bm u^c|-1}_r} 1
 \d \mu_{d-|\bm a|}\left(\frac{\x_{\bm u^c}}{\sqrt{1-\norm{\y_{\bm a}}^2}}\right) \d \frac{ \x_{\bm u\backslash \bm a}}{\sqrt{1-\norm{\y_{\bm a}}^2}}\\
 &\stackrel{ \eqref{eq:int_ball}}{=} \frac{1}{\omega_{d-|\bm a|}} \int_{\B^{|\bm u\backslash \bm a|}} \frac{1}{r}g_{\bm u}\left(\x_{\bm u}\right)\omega_{d-|\bm u|} r^{d-|\bm u|}\d \frac{ \x_{\bm u\backslash \bm a}}{\sqrt{1-\norm{\y_{\bm a}}^2}}\\
 &= \frac{\omega_{d-|\bm u|} }{\omega_{d-|\bm a|}} \int_{\B^{|\bm u\backslash \bm a|}} r^{d-|\bm u|-1}g_{\bm u}\left(\x_{\bm u}\right)\d \frac{ \x_{\bm u\backslash \bm a}}{\sqrt{1-\norm{\y_{\bm a}}^2}}\\
 &= \frac{\omega_{d-|\bm u|} }{\omega_{d-|\bm a|}} \changed{\left(1-\norm{\y_{\bm a}}^2\right)^{(-d+|\bm u|+1)/2}}\int_{\B^{|\bm u\backslash \bm a|}} \changed{\left(1-\norm{\x_{\bm u}}^2\right)^{(d-|\bm u|-1)/2}}g_{\bm u}\left(\x_{\bm u} \right)\d \frac{ \x_{\bm u\backslash \bm a}}{\sqrt{1-\norm{\y_{\bm a}}^2}}\\
    &= \frac{\omega_{d-|\bm u|} }{\omega_{d-|\bm a|}} \changed{\left(1-\norm{\y_{\bm a}}^2\right)^{(-d+|\bm u|+1)/2}}\int_{\B^{|\bm u\backslash \bm a|}}  \changed{\left(\left(1-\norm{\y_{\bm a}}^2\right)\left( 1-\norm{\x_{\bm u\backslash \bm a}}^2\right)\right)^{(d-|\bm u|-1)/2}}g_{\bm u}\left(\y_{\bm a}, \sqrt{1-\norm{\y_{\bm a}}^2}\x_{\bm u\backslash \bm a} \right)\d  \x_{\bm u\backslash \bm a}\\
     &= \frac{\omega_{d-|\bm u|} }{\omega_{d-|\bm a|}} \int_{\B^{|\bm u\backslash \bm a|}}  \changed{\left( 1-\norm{\x_{\bm u\backslash \bm a}}^2\right)^{(d-|\bm u|-1)/2}}g_{\bm u}\left(\y_{\bm a}, \sqrt{1-\norm{\y_{\bm a}}^2}\x_{\bm u\backslash \bm a} \right)\d  \x_{\bm u\backslash \bm a}\\
      &= \frac{\omega_{d-|\bm u|} }{\omega_{d-|\bm a|} } \int_{\B^{|\bm u\backslash \bm a|}}  \changed{\left( 1-\norm{\x_{\bm u\backslash \bm a}}^2\right)^{(d-|\bm u|-1)/2}}g_{\bm u}\left(\y_{\bm a}, \sqrt{1-\norm{\y_{\bm a}}^2}\x_{\bm u\backslash \bm a} \right)\d  \x_{\bm u\backslash \bm a}.
\end{align*}
\end{small}
\end{allowdisplaybreaks}
For arbitrary $\bm v\in \D$, which does not necessarily has to be a subset of $\bm u$, 
let $\bm a = \bm u \cap \bm v$ \changed{and set $r \coloneqq\sqrt{1-\frac{\norm{\x_{\bm u \backslash \bm a}}^2}{1-\norm{\y_{\bm v}}^2}}$, then }
\begin{allowdisplaybreaks}
\begin{small}
\begin{align*}
 &[P_{\bm v}g_{\bm u}](\y_{\bm v}) = \frac{1}{m_{\bm v}(\y_{\bm v})}\,\int_{\S^{|\bm v^c|-1}_{\sqrt{1-\norm{\y_{\bm v }}^2}} } 
g_{\bm u}(\x_{\bm u}) \d \mu_{|\bm v^c|-1}(\x_{\bm v^c})\\
&\quad\stackrel{\eqref{eq:split_integral_r}}{=} 
\frac{\changed{\left(1-\norm{\y_{\bm v}}^2\right)^{(|\bm v^c|-1)/2}}}{m_{\bm v}(\y_{\bm v})}  \int_{\B^{|\bm u \backslash \bm a|}}\frac{1}{r}
\int_{\S^{|(\bm v \cup \bm u)^c|-1}_{r}}  g_{\bm u}\left(\x_{\bm u}\right) \d\mu_{|(\bm u\cup \bm v)^c|-1}\left(\frac{\x_{(\bm u\cup \bm v)^c}}{\sqrt{1-\norm{\y_{\bm v}}^2}}\right) \d \left(\frac{\x_{\bm u \backslash \bm a}}{\sqrt{1-\norm{\y_{\bm v}}^2}}\right)\\
&\quad= 
\frac{\changed{\left(1-\norm{\y_{\bm v}}^2\right)^{(|\bm v^c|-1)/2}}}{m_{\bm v}(\y_{\bm v})}  \int_{\B^{|\bm u \backslash \bm a|}}\frac{g_{\bm u}\left(\x_{\bm u}\right)}{r}
\int_{\S^{|(\bm v \cup \bm u)^c|-1}_{r}}   \d\mu_{|(\bm u\cup \bm v)^c|-1}\left(\frac{\x_{(\bm u\cup \bm v)^c}}{\sqrt{1-\norm{\y_{\bm v}}^2}}\right) \d \left(\frac{\x_{\bm u \backslash \bm a}}{\sqrt{1-\norm{\y_{\bm v}}^2}}\right)\\
&\quad= 
\frac{\omega_{|\bm u\cup \bm v|^c -1}}{m_{\bm v}(\y_{\bm v})}  \changed{\left(1-\norm{\y_{\bm v}}^2 \right)^{(|\bm v^c|-1)/2} }   \int_{\B^{|\bm u \backslash \bm a|}}g_{\bm u}\left(\x_{\bm u}\right)r^{|\bm u\cup \bm v|^c -2} 
 \d \left(\frac{\x_{\bm u \backslash \bm a}}{\sqrt{1-\norm{\y_{\bm v}}^2}}\right)\\
 &\quad= 
\frac{\omega_{|\bm u\cup \bm v|^c -1}}{m_{\bm v}(\y_{\bm v})}  \changed{\left(1-\norm{\y_{\bm v}}^2 \right)^{(|\bm v^c|-1)/2}} \int_{\B^{|\bm u \backslash \bm a|}}\changed{\left(1-\norm{\x_{\bm u \backslash \bm a}}^2\right)^{|\bm u\cup \bm v|^c/2 -1}} g_{\bm u}\left(\y_{\bm a}, \sqrt{1-\norm{\y_{\bm v}}^2}\,\x_{\bm u\backslash \bm a}\right)
 \d \x_{\bm u \backslash \bm a}\\
  &\quad= 
\frac{\omega_{|\bm u\cup \bm v|^c -1}}{\omega_{d-|\bm v|}}  \sqrt{1-\norm{\y_{\bm v}}^2 } \int_{\B^{|\bm u \backslash \bm a|}}\changed{\left(1-\norm{\x_{\bm u \backslash \bm a}}^2\right)^{|\bm u\cup \bm v|^c/2 -1}} g_{\bm u}\left(\y_{\bm a}, \sqrt{1-\norm{\y_{\bm v}}^2}\,\x_{\bm u\backslash \bm a}\right)
 \d \x_{\bm u \backslash \bm a}.
\end{align*}
\end{small}
\end{allowdisplaybreaks}

    This finishes the proof. 
\end{proof}

\subsection{The ANOVA operator and the integral conditions}\label{sec:ANOVA_operator}
Every function $f\in L_2(\S^d)$ can be decomposed arbitrarily in the form~\eqref{eq:anova_sphere}, if we choose arbitrary terms for $f_{\bm u}$ for $|\bm u|\leq d-1$ and simply choose the remainder that is missing from $f$ as the term $f_{[d+1]}(\x)$. For the classical ANOVA decomposition, orthogonality \changed{of all the ANOVA terms~\eqref{eq:orth_classic}} is required to get a unique decomposition. This goes hand in hand with the condition~\eqref{eq:integal_cond_classic} that certain mean values of the ANOVA terms are zero. In the following we propose similar integral conditions also for the spherical ANOVA decomposition.

\begin{definition}\label{def:ANOVA_operator}
For a function $f\in L_2(\S^d)$, we introduce \changed{ the ANOVA operators $\AA_{\bm u}$ iteratively by} 
\begin{align*}
\AA_{\varnothing} f&= P_{\varnothing} f = \frac{1}{\omega_d}\int_{\S^{d}} f(\x) \d \mu_d(\x),\notag\\
[\AA_{\bm u} f](\x_{\bm u}) &=\left[P_{\bm u}f\right](\x_{\bm u}) - \sum_{\bm v\subset \bm u} [\AA_{\bm v} f](\x_{\bm v}) &\text{ if  } 1\leq |\bm u|\leq  d-1,\\
[\AA_{[d+1]} f](\x) &=f(\x) - \sum_{\bm v\in \D\backslash[d+1]} [\AA_{\bm v} f](\x_{\bm v}). & \notag
\end{align*}
In this definition the terms 
$\AA_{\bm u}f\colon \B^{|\bm u|}\rightarrow\R$ depend only on the variables $\x_{\bm u}$.
\end{definition}
This definition is similar to the definition of the classical ANOVA decomposition, see~\eqref{eq:moebius_classical}.
Another formula for the \changed{ANOVA operators $\AA_{\bm u}$}, which follows by induction directly from the iterative definition of the operator $\AA_{\bm u}$ is
\begin{equation}\label{eq:moebius_eq}
\AA_{\bm u} f = \sum_{\bm v\subseteq \bm u}(-1)^{|\bm u|-|\bm v|} P_{\bm v} f.
\end{equation}
Additionally, all terms sum up to the original function, and we obtain the decomposition on the sphere, 
\begin{equation}\label{eq:sum_AA}
f(\x) = \sum_{\bm u\in \D}[\AA_{\bm u} f](\x_{\bm u}).
\end{equation}
A closer look at Definition~\ref{def:ANOVA_operator} and equation~\eqref{eq:sum_AA} reveals that, for the definition to be well-defined, for any function $f$ consisting of a single term, $f=\AA_{\bm u} f $, it must hold that
$$\AA_{\bm a} f = \begin{cases}
    f& \text{ if } \bm u=\bm a,\\
    0 & \text{ otherwise. } 
\end{cases}$$ 
This automatically means by~\eqref{eq:moebius_eq} that for $\bm a\neq \varnothing$ and a function $f$ with $f=\AA_{\bm u} f $,
\begin{equation}\label{eq:P_af_u}
\left[P_{\bm a} \AA_{\bm u}f\right] (\y_{\bm a}) 
=\begin{cases}
 [\AA_{\bm u}f](\y_{\bm u})&\text{ if }\bm a \supseteq \bm u \\
0 &\text{ in all other cases}.\\ 
\end{cases}
    \end{equation}
\changed{This condition is similar to the integral conditions~\eqref{eq:integal_cond_classic} to ensure uniqueness in the classical ANOVA decomposition.} 

The terms defined in Definition~\ref{def:ANOVA_operator} are orthogonal with respect to the uniform measure on the sphere $\S^d$, which the following theorem shows.
\begin{theorem}
For $\bm u,\bm v\in \D$ and $\bm v \neq \bm u$ the terms $\AA_{\bm u}f$ and $\AA_{\bm v}f$, defined in  Definition~\ref{def:ANOVA_operator}, are orthogonal with respect to the surface measure on $\S^d$,
$$\langle \AA_{\bm u}f,\AA_{\bm v}f\rangle_{\S^d} 
= 0.$$
\end{theorem}
\begin{proof}
Set $\bm w = \bm u \cap \bm v$. If $\bm  w=\varnothing$, we skip the following first step and proceed with the \changed{nested} integral. Otherwise let without loss of generality $\bm u \backslash \bm w \neq  \varnothing$, then
\begin{align*}
    \langle \AA_{\bm u}f,\AA_{\bm v}f\rangle_{\S^d} 
    &= \int_{\S^d } [\AA_{\bm u}f]\left(\x_{\bm u}\right) [\AA_{\bm v}f]\left(\x_{\bm v}\right) \d \mu_{d} (\x)\\
    &\stackrel{\eqref{eq:split_integral}}{=}\int_{\B^{|\bm w|}}\frac{1}{\sqrt{1-\norm{\x_{\bm w}}^2}}\int_{ \S^{|\bm w^c|-1}_{\sqrt{1-\norm{\x_{\bm w}}^2}}}  [\AA_{\bm u}f]\left(\x_{\bm u}\right) [\AA_{\bm v}f]\left(\x_{\bm v}\right) \d\mu_{d-|\bm  w|}(\x_{\bm w^c}) \d \x_{\bm w}.
\end{align*}
For the \changed{nested} integral \changed{we denote $b \coloneqq\sqrt{1-\frac{\norm{\x_{\bm u\backslash \bm w}}^2}{1-\norm{\x_{\bm w}}^2}} = \sqrt{\frac{1-\norm{\x_{\bm u}}^2}{1-\norm{\x_{\bm w}}^2}}$} and receive
\begin{small}
\begin{align*}
&\int_{ \S^{|\bm w^c|-1}_{\sqrt{1-\norm{\x_{\bm w}}^2}}}  [\AA_{\bm u}f]\left(\x_{\bm u}\right) [\AA_{\bm v}f]\left(\x_{\bm v}\right) \d\mu_{d-|\bm  w|}(\x_{\bm w^c}) \\
&\quad\stackrel{\eqref{eq:split_integral_r}}{=} \changed{\left(1-\norm{\x_{\bm w}}^2\right)^{(|\bm w^c|-1)/2}} \int_{\B^{|\bm u\backslash\bm  w|}} \frac{1}{b} \int_{\S^{|\bm u^c|-1}_{b}}[\AA_{\bm u}f]\left(\x_{\bm u}\right) [\AA_{\bm v}f]\left(\x_{\bm v}\right) \d\mu_{|\bm u^c|-1}\left(\frac{\x_{\bm u^c}}{\sqrt{1-\norm{\x_{\bm w}}^2}}\right)\d \frac{\x_{\bm u\backslash \bm w}}{\sqrt{1-\norm{\x_{\bm w}}^2}}\\
&\quad= \changed{\left(1-\norm{\x_{\bm w}}^2\right)^{(|\bm w^c|-1)/2}} \int_{\B^{|\bm u\backslash\bm  w|}} \frac{1}{b} [\AA_{\bm u}f]\left(\x_{\bm u}\right)\int_{\S^{|\bm u^c|-1}_{b}} [\AA_{\bm v}f]\left(\x_{\bm v}\right) \d\mu_{|\bm u^c|-1}\left(\frac{\x_{\bm u^c}}{\sqrt{1-\norm{\x_{\bm w}}^2}}\right)\d \frac{\x_{\bm u\backslash \bm w}}{\sqrt{1-\norm{\x_{\bm w}}^2}}.
\end{align*}
\end{small}
Again, we proceed with the \changed{nested} integral by
\begin{allowdisplaybreaks}
\begin{small}
\begin{align*}
\int_{\S^{|\bm u^c|-1}_{b}} [\AA_{\bm v}f]\left(\x_{\bm v}\right) \d\mu_{|\bm u^c|-1}\left(\frac{\x_{\bm u^c}}{\sqrt{1-\norm{\x_{\bm w}}^2}}\right)
&=\changed{\left(1-\norm{\x_{\bm w}}^2\right)^{(1-|\bm u^c|)/2}}\int_{\S^{|\bm u^c|-1}_{\sqrt{1-\norm{\x_{\bm u }}^2}} } 
[\AA_{\bm v}f](\x_{\bm v}) \d \mu_{|\bm u^c|-1}(\x_{\bm u^c})\\
 &= \changed{\left(1-\norm{\x_{\bm w}}^2\right)^{(1-|\bm u^c|)/2}}\,\frac{1}{m_{\bm u}(\y_{\bm u })}\,[ P_{\bm u} [\AA_{\bm v}f]](\x_{\bm w}) =0.
\end{align*}
\end{small}
\end{allowdisplaybreaks}
These terms are zero due to \eqref{eq:P_af_u}.
\end{proof}

For the classical ANOVA decomposition with tensor product structure, \changed{we have the integral conditions~\eqref{eq:integal_cond_classic}} due to the tensor product structure~\eqref{eq:P_u_classic} of the operator $P_{\bm u}$. On the sphere, these conditions are slightly more complicated since the projection operator $P_{\bm u}$ is not the product of one-dimensional projections. 

\begin{lemma}\label{lem:integral_cond}
The operators $\AA_{\bm u}$ in Definition~\ref{def:ANOVA_operator} fulfill condition~\eqref{eq:P_af_u} if the integral conditions
\begin{align}\label{eq:integral_conditions_d1}
0&= \int_{\B^{|\bm u\backslash \bm a|}}  \changed{\left( 1-\norm{\x_{\bm u\backslash \bm a}}^2\right)^{(d-|\bm u|-1)/2}}[\AA_{\bm u}f]\left(\y_{\bm a}, \sqrt{1-\norm{\y_{\bm a}}^2}\,\x_{\bm u\backslash \bm a} \right)\d  \x_{\bm u\backslash \bm a} \\
&\quad \quad \quad  \text{ for all }   \bm a\subset \bm u, \y_{\bm a}\in \B^{|\bm a|} \text{ \changed{and}}\notag\\
0&= \sqrt{1-\norm{\y_{\bm a}}^2 -r^2}  \int_{\B^{|\bm u \backslash \bm a|}}\changed{\left(1-\norm{\x_{\bm u \backslash \bm a}}^2\right)^{m/2}} [\AA_{\bm u}f]\left(\y_{\bm a}, \sqrt{1-\norm{\y_{\bm a}}^2-r^2}\,\x_{\bm u\backslash \bm a}\right)
 \d \x_{\bm u \backslash \bm a},\label{eq:integral_conditions_d2} \\
 &  \quad \quad \quad  \text{ for all } \bm a\subset \bm u, \y_{\bm a}\in \B^{|\bm a|}, 0\leq r\leq \sqrt{1-\norm{\y_{\bm a}}^2},  m = -1,0,1, \ldots, d-|\bm u|-2 \notag 
\end{align}
are fulfilled.
\end{lemma}
\begin{proof}
In Lemma~\ref{lem:Pafu=} we applied the projection operator $P_{\bm a}$ to a function, which depends only on the variables $\x_{\bm u}$. The function $\AA_{\bm u} f $ is such a function. Hence, this lemma follows directly from Lemma~\ref{lem:Pafu=}.  
\end{proof}
These conditions are similar to the integral conditions~\eqref{eq:integal_cond_classic} for the classical ANOVA decomposition, where we fix one $\y_{\bm a}$ in $f_{\bm u}(\x_{\bm u})$ and integrate over all other variables. These integrals then vanish for all $\y_{\bm a}$. Shortly, the ANOVA terms have to fulfill~\eqref{eq:P_af_u}. For the classical ANOVA decomposition due to the tensor product structure~\eqref{eq:P_u_classic}, it is enough to demand $P_{\{i\}}f_{\bm u}$ for all $i\in\bm u$. In the spherical case, this is much more complicated and we have to demand the integral conditions in Lemma~\ref{lem:integral_cond}. As we will see in the next subsection, only odd functions $\AA_{\bm u}f$ fulfill these conditions.

\subsection{The parity of the terms \texorpdfstring{$\AA_{\bm u} f$}{} }\label{sec:parity}
In Definition~\ref{def:ANOVA_operator} we defined operators $\AA_{\bm u}$, which split a spherical function into lower dimensional terms. In Lemma~\ref{lem:integral_cond} we showed integral conditions that the terms $\AA_{\bm u}f$ have to fulfill. In fact, only odd functions fulfill these conditions.
\begin{theorem}
 Let a function $g\in L_2(\B^n)$, where $n=|\bm u| \leq d-1$.
 That the function $g$ fulfills the integral conditions \eqref{eq:integral_conditions_d1} and \eqref{eq:integral_conditions_d2} is equivalent to the function $g$ being odd in every direction.
\end{theorem}
\begin{proof}
If the function $g$ is odd, it easily follows that \eqref{eq:integral_conditions_d1} and \eqref{eq:integral_conditions_d2} hold true, so one direction follows directly.\par
For the other direction, let $g\in L_2(\B^n)$ be arbitrary and fulfill the integral conditions \eqref{eq:integral_conditions_d1} and \eqref{eq:integral_conditions_d2}. 
Then, decompose $g$ into its even and odd parts $g_{\bm \xi}$ like in~\eqref{eq:def_even_odd_d}.\par
Let us fix $i\in \{1,\ldots, n\}$. Then, \eqref{eq:integral_conditions_d2} for $\y_{[n]\backslash i}$ and $m=0$ states that 
$$0= \sqrt{1-\norm{\y_{[n]\backslash i}}^2 -r^2}  \int_{-1}^1g\left(\y_{[n]\backslash i}, \sqrt{1-\norm{\y_{[n]\backslash i}}^2-r^2}\,x_i\right)
 \d x_i$$
 for all $0\leq r\leq \sqrt{1-\norm{\y_{[n]\backslash i}}^2}$.
Using the decomposition~\eqref{eq:def_even_odd_d} of $g$ into odd and even parts, in this integral of the terms $g_{\bm \xi}$ with $\xi_i = 1$ are zero, such that we have
\begin{align*}
    0&= \sqrt{1-\norm{\y_{[n]\backslash i}}^2 -r^2}  \int_{-1}^1 \sum_{\stackrel{\bm \xi\in \{0,1\}^{n}}{\xi_i=0}}g_{\bm \xi}\left(\y_{[n]\backslash i}, \sqrt{1-\norm{\y_{[n]\backslash i}}^2-r^2}\,x_i\right)
 \d x_i\\
 &= \int_{-\sqrt{1-\norm{\y_{[n]\backslash i}}^2-r^2}}^{\sqrt{1-\norm{\y_{[n]\backslash i}}^2-r^2}} \sum_{\stackrel{\bm \xi\in \{0,1\}^{n}}{\xi_i=0}}g_{\bm \xi}\left(\y_{[n]\backslash i}, x_i\right)
 \d x_i.
\end{align*}
That means, that the function $G\coloneqq\sum_{\stackrel{\bm \xi\in \{0,1\}^{n}}{\xi_i=0}}g_{\bm \xi}$ is orthogonal to all step functions of the form $\mathbf{1}_{[-r,r]}(x_i)$. The symmetric indicator functions of arbitrary step functions are a linear combination of the indicator functions of the form $\mathbf{1}_{[-r,r]}(t)$. Furthermore, it follows that $G$ is also orthogonal to all step functions on $\B^n$, which are supported on a tensor-product domain and are even in $x_i$ direction. Since step functions are dense in $L_2(\B^n)$, it follows that
\[
\langle G, h\rangle_{L_2(\B^n)} = 0 \quad \text{for all even } h\in L_2(\B^n) \text{ with } h(\x) = h(\x_{[n]\backslash i}, x_i) \text{\,a.e.}
\]
Hence $G$ must vanish almost everywhere. Inserting the definitions of the functions $g_{\bm \xi}$ from~\eqref{eq:def_even_odd_d} into $G$ gives then 
\begin{allowdisplaybreaks}
\begin{align*}
    0&= \sum_{\stackrel{\bm \xi\in \{0,1\}^{n}}{\xi_i=0}}g_{\bm \xi}(\x) =\sum_{\stackrel{\bm \xi\in \{0,1\}^{n}}{\xi_i=0}} \tfrac{1}{2^{n}} \sum_{\bm k \in \{-1,1\}^{n}} \left(\prod_{\ell=1}^{n}k_\ell^{\xi_\ell}\right) g(\bm k \odot \x )\\
     &= \tfrac{1}{2^{n}}\sum_{\stackrel{\bm \xi\in \{0,1\}^{n}}{\xi_i=0}}  \sum_{\bm k \in \{-1,1\}^{n}} \left(\prod_{\ell=1}^{n}k_\ell^{\xi_\ell}\right) g(\bm k \odot \x )\\
      &= \tfrac{1}{2^{n}}   \sum_{\bm k \in \{-1,1\}^{n}} g(\bm k \odot \x )\sum_{\stackrel{\bm \xi\in \{0,1\}^{n}}{\xi_i=0}}\left(\prod_{\ell=1}^{n}k_\ell^{\xi_\ell}\right)\\
         &= \tfrac{1}{2^{n}}   \sum_{\bm k \in \{-1,1\}^{n}} g(\bm k \odot \x )\sum_{\bm \xi_{[n]\backslash i}\in \{0,1\}^{n-1}}\left(\prod_{\ell\neq i}^{}k_\ell^{\xi_\ell}\right) \\
      &= \tfrac{1}{2^{n}} \left(  \underbrace{2^{n}g(\x)}_{k=\bm 1}  +  \underbrace{2^{n}g(\x_{[n]\backslash i}, -x_i)}_{k=\bm 1 - 2\bm e_i} \right),
\end{align*}
\end{allowdisplaybreaks}%
since all the other terms cancel out. It follows directly that 
 $g (\x) = -g(\x_{[n]\backslash i}, -x_i) $.
Similarly, this is true for all other $i\in \{1,\ldots, n\}$, such that $g$ is odd in all directions.
\end{proof}

\subsection{The definition of the spherical ANOVA decomposition}\label{sec:ANOVA_sphere}
While the decomposition into even and odd parts is straightforward in Euclidean space, the algebraic constraint of the sphere introduces specific dependencies between the variables. Here we discuss how multivariate parity impacts the linear independence of the basis terms and identify which symmetry patterns lead to redundant representations.\par
The ANOVA operators $\AA_{\bm u}$ from Definition~\ref{def:ANOVA_operator} produce only odd terms. This means, in the spherical ANOVA decomposition~\eqref{eq:anova_xi_sphere} only terms of the form $f_{\bm u,\bm \xi}$ with $\bm \xi = (\bm 1_{\bm u}, \bm 0_{\bm u^c})$ are produced by applying $\AA_{\bm u}$ to the function $f$ itself and all even parts of a function $f$ are contained in the last term $\AA_{[d+1]}f$.\par

To include the even terms also in the lower-dimensional terms of the spherical ANOVA decomposition, we propose to define a spherical ANOVA decomposition of the form~\eqref{eq:anova_xi_sphere} by the following.
\begin{definition}\label{def:spherical_anova_final}
For a function $f\in L_2(\S^d)$ we use the decomposition into even and odd parts as defined in~\eqref{eq:def_even_odd_d} and the projection operator defined in~\eqref{eq:Pu_def}. With these ingredients we introduce the spherical ANOVA decomposition of the form~\eqref{eq:anova_xi_sphere} by 
\begin{align*}
\intertext{\textbf{1. The constant term:}}
f_{\varnothing, \bm 0} &\coloneqq P_{\varnothing} f = \frac{1}{\omega_d}\int_{\S^{d}} f(\x) \d \mu_d(\x).\\
\intertext{\textbf{2. The mixed terms:}
For all $\bm v\in \D\backslash \varnothing$ and $\bm u \supseteq \bm v$ set $\bm \xi=(\bm 1_{\bm v},  \bm 0_{\bm v^c})\in \{0,1\}^{d+1}$. If $d+1\notin \bm v$, then $d+1\notin \bm u$ and let }
 f_{\bm u,(\bm 1_{\bm v},\bm 0_{\bm v^c})}(\x_{\bm u})&\coloneqq P_{\bm u}\left[\Xi_{(\bm 1_{\bm v},\bm 0_{\bm v^c})}f(\x_{\bm u})\right] - \sum_{\bm v\subseteq \bm a\subset \bm u} f_{\bm a,(\bm 1_{\bm v},\bm 0_{\bm v^c})}(\x_{\bm a}) &\text{ if  } |\bm v|\leq |\bm u|\leq  d-1,\\
f_{[d+1],(\bm 1_{\bm v},\bm 0_{\bm v^c})}(\x) &=\left[\Xi_{(\bm 1_{\bm v},\bm 0_{\bm v^c})}f\right](\x) - \sum_{\bm v\subseteq\bm a\in \D\backslash[d+1]} f_{\bm a,(\bm 1_{\bm v},\bm 0_{\bm v^c})} (\x_{\bm a}). \\
\intertext{\textbf{3. The pure even terms:} For all $\bm u \subset [d]\backslash  \varnothing$ let  }
     f_{\bm u, \bm 0}(\x_{\bm u}) &\coloneqq P_{\bm u}\left[\Xi_{\bm 0}f\right](\x_{\bm u})  - \sum_{\bm a \subsetneq \bm u} f_{\bm a, \bm 0}(\x_{\bm a}).
\end{align*}
\end{definition}
We will discuss in Section~\ref{sec:redundancy} why we do not include the terms where $d+1\in \bm u$ with $\xi_{d+1}=0$. 
Our definition of spherical ANOVA terms is of the form~\eqref{eq:anova_xi_sphere}, all the terms $f_{\bm u,\bm \xi}$ have parity $\bm \xi$ and 
$$ \Xi_{\bm \xi} f (\x) =  \sum_{\stackrel{\bm a\in \D}{\supp \bm \xi \subseteq \bm a}} f_{\bm a,\bm \xi}(\x_{\bm a}).$$
In Example~\ref{ex:d=2} we considered the case $d=2$. A visualization of all the possible indices $\bm u$ and corresponding $\bm \xi$ for the terms $f_{\bm u,\bm\xi}$ for $d=3$ can be found in Figure~\ref{fig:diagram}. There we show the totality of all terms $f_{\bm u,\bm \xi}$. Later, in Section~\ref{sec:redundancy} we will explain which of these terms we have to omit due to redundancy on the sphere. The terms in Definition~\ref{def:spherical_anova_final} can be calculated more easily by the following.
\begin{lemma}
The Definition~\ref{def:spherical_anova_final} is equivalent to this definition of spherical ANOVA terms,
\begin{align}
\intertext{\textbf{1. The constant term:}}
f_{\varnothing, \bm 0} &=P_{\varnothing} \Xi_{\bm 0} f = \frac{1}{\omega_d}\int_{\S^{d}} f(\x) \d \mu_d(\x).\notag\\
\intertext{\textbf{2. The mixed terms:} For all $\varnothing\neq\bm v\subseteq [d+1]$ and $\bm u \supseteq \bm v$ set $\bm \xi=(\bm 1_{\bm v}, \bm 0_{\bm v^c})\in \{0,1\}^{d+1}$. If $d+1\notin \bm v$, then $d+1\notin \bm u$ and let }
f_{\bm u,(\bm 1_{\bm v},\bm 0_{\bm v^c})}(\x_{\bm u})&\coloneqq \sum_{\bm v\subseteq \bm a\subseteq \bm u} (-1)^{|\bm u|-|\bm a|}P_{\bm a} [\Xi_{(\bm 1_{\bm v},\bm 0_{\bm v^c})} f ].  \label{eq:f_u_xi_easy}\\
\intertext{\textbf{3. The pure even terms:}
 For all $\bm u \subset [d]\backslash  \varnothing$ let  }
     f_{\bm u, \bm 0}(\x_{\bm u}) &= \sum_{\bm a\subseteq \bm u}(-1)^{|\bm u|-|\bm a|}P_{\bm u}\left[\Xi_{\bm 0}f\right](\x_{\bm u}) . \notag
\end{align}
\end{lemma}
\begin{proof}
This proof is similar to the proof of~\eqref{eq:moebius_classical} for the classical ANOVA decomposition using induction:
For $\bm u = \bm v$ the sum in~\eqref{eq:f_u_xi_easy} is just one summand, namely $P_{\bm v} [\Xi_{(\bm 1_{\bm v},\bm 0_{\bm v^c})} f ]$. For the induction step let~\eqref{eq:f_u_xi_easy} be fulfilled for all $\bm u\supseteq \bm v$ with $|\bm u|\leq k$ and choose $\bm u\supseteq \bm v$ with $|\bm u|=k+1$. Then, due to the alternating sum of binomial coefficients, 
\begin{align*}
f_{\bm u,(\bm 1_{\bm v},\bm 0_{\bm v^c})}(\x_{\bm u})& = P_{\bm u}\left[\Xi_{(\bm 1_{\bm v},\bm 0_{\bm v^c})}f(\x_{\bm u})\right] - \sum_{\bm v\subseteq \bm a\subset \bm u} f_{\bm a,(\bm 1_{\bm v},\bm 0_{\bm v^c})}(\x_{\bm a})\\
&=P_{\bm u}\left[\Xi_{(\bm 1_{\bm v},\bm 0_{\bm v^c})}f(\x_{\bm u})\right]  - \sum_{\bm v\subseteq \bm a\subset \bm u} \sum_{\bm v\subseteq \bm b\subseteq \bm a} (-1)^{|\bm a|-|\bm b|}P_{\bm b} [\Xi_{(\bm 1_{\bm v},\bm 0_{\bm v^c})} f ]\\
&=P_{\bm u}\left[\Xi_{(\bm 1_{\bm v},\bm 0_{\bm v^c})}f(\x_{\bm u})\right]  - \sum_{\bm v\subseteq \bm b\subseteq \bm u} \sum_{k=0}^{|\bm u\backslash \bm b|-1}\binom{|\bm u\backslash\bm b|}{k}  (-1)^{k}P_{\bm b} [\Xi_{(\bm 1_{\bm v},\bm 0_{\bm v^c})} f ]\\
&=P_{\bm u}\left[\Xi_{(\bm 1_{\bm v},\bm 0_{\bm v^c})}f(\x_{\bm u})\right]  + \sum_{\bm v\subseteq \bm b\subseteq \bm u}  (-1)^{|\bm u|-|\bm b|}P_{\bm b} [\Xi_{(\bm 1_{\bm v},\bm 0_{\bm v^c})} f ]\\
&=\sum_{\bm v\subseteq \bm b\subseteq \bm u} (-1)^{|\bm u|-|\bm b|}P_{\bm b} [\Xi_{(\bm 1_{\bm v},\bm 0_{\bm v^c})} f ].
\end{align*}
For the pure even case, the proof is the same with $\bm v = \varnothing$.
\end{proof}

In our spherical ANOVA decomposition we have separated the parities and established the hierarchy. Our proposed definition of a spherical ANOVA decomposition has similarities to the classical ANOVA decomposition, whereby we cannot transfer all properties of the classical ANOVA decomposition to the sphere. See Section~\ref{sec:numerics} for example functions and their decomposition into spherical ANOVA terms. \par

In our decomposition we have some orthogonality for the spherical ANOVA terms, which we will show in the following section.

\subsection{Orthogonality and Redundancy}\label{sec:redundancy}
The decomposition into parity parts belonging to the parity vectors $\bm\xi\in \{0,1\}^{d+1}$, see also~\eqref{eq:def_even_odd_d}, is an orthogonal decomposition, which also means that for the spherical ANOVA terms defined in Definition~\ref{def:spherical_anova_final} we have the orthogonality for all $\bm u,\tilde{\bm u}$ with $\bm u\neq\tilde{\bm u}$,
\begin{equation}\label{eq:orth_xi}
\langle f_{\bm u,\bm \xi} ,f_{\tilde{\bm u},\tilde{\bm \xi}}\rangle_{\S^d} = 0 \quad \text{ if } \bm \xi\neq \tilde{\bm \xi}.
\end{equation}
Unfortunately, we do not have orthogonality for spherical ANOVA terms with the same parity vector $\bm \xi$.
Even worse, the equation $\sum_{i=1}^{d+1}x_i^2 = 1$ on the sphere naturally leads to redundancy of the even terms. Every even function $h(x_{d+1})$ can also be written as a $\tilde h(x_{d+1}^2) = \tilde h (1-x_1^2-\cdots - x_{d}^2)$, which is a function depending on the variables $x_1$ to $x_d$. Additionally, $\tilde h(x_1,\ldots,x_d)$ is even in all directions. \par
For that reason we propose to omit some terms in the Definition~\ref{def:spherical_anova_final}. Namely by the substitution 
$$x_{d+1}^2 = 1- \sum_{i=1}^d x_i^2 $$
we propose in Definition~\ref{def:spherical_anova_final} to omit all terms, which depend on $x_{d+1}$ and where the dependence is even, i.e.,
$$f_{\bm u,\bm \xi}= 0 \quad\text{ if } d+1\in \bm u \text{ and } \xi_{d+1} = 0.$$
Instead of $x_{d+1}$ one could also choose one other specific variable. For simplicity, we will stick to $x_{d+1}$ in the following. In Figure~\ref{fig:diagram}
we illustrate the omitted terms for the case $d=3$.
In total this means, the number of non-zero terms in the decomposition~\eqref{eq:anova_xi_sphere} reduces to
\begin{align}\label{eq:number_spherical_terms2}
&3^{d+1} -2^d\,(d+1) - \sum_{i=0}^{d-2} \binom{d}{i} 2^{i} - 2^{d} = 3^{d+1} -2^d \,(d+2) - (1+2)^d + \binom{d}{d-1}2^{d-1}+2^d  \notag\\
&= 2\cdot3^{d} -2^d \,(d+2)  + d \, 2^{d-1}+2^d = 2\cdot 3^{d} -2^{d-1}(d+2).
\end{align}

\subsection{The challenges with the even terms}\label{sec:challenge}
The aim of this paper is to construct a spherical ANOVA decomposition of the form~\eqref{eq:anova_xi_sphere}, which combines a decomposition into ANOVA terms with the decomposition into parity parts. In Section~\ref{sec:parity} we showed that the terms $\AA_{\bm u}$ should be odd in all directions $i\in \bm u$. This means that with the operator $\AA_{\bm u}$ we could only represent functions $f$, which have non-zero spherical ANOVA terms only for $\bm u$ with $\bm \xi=(\bm 1_{\bm u},\bm 0_{\bm u^c})$. We also include even parts of the function $f$ in our decomposition in Definition~\ref{def:spherical_anova_final}. We analyze these parts here in more detail. \par
If the function $f$ is a sum of terms, which are odd in all direction, the Definition~\ref{def:spherical_anova_final} coincides with the decomposition, which we would expect from the classical ANOVA decomposition: Consider for example
$$f(\x) =  \underbrace{x_1  x_2^3}_{f_{\{1,2\},\bm e_1 +\bm e_2}} + \underbrace{2 x_3  x_4^5}_{f_{\{3,4\},\bm e_3 +\bm e_4}} + \underbrace{0.05x_5}_{f_{\{5\},\bm e_5}},$$
for more details about the calculations see Appendix~\ref{sec:app1}. We marked in Figure~\ref{fig:diagram} all the terms, which are odd in all directions. We do not need to do anything else for these terms.\par

But as soon as some even term is involved, we run into challenges. We begin with an easy example which shows the problems of the underlying dependence of the variables on the sphere. 
\begin{figure}[p]
\begin{tikzpicture}[
    node distance=2.5cm and 2.2cm,
    box/.style={
        rectangle,
        rounded corners=3pt,
        draw=black!80,
        very thick,
        align=center,
        minimum width=2cm,
        minimum height=1.3cm,
        fill=white,
        drop shadow,
        font=\scriptsize\sffamily
    },
    link/.style={->, >=LaTeX, thick, color=gray!30},
]
    
    \newcommand{\vEmpty}[1]{\textcolor{blue!80!black}{\textbf{#1}}}
    \newcommand{\vOne}[1]{\textcolor{red!90!black}{\textbf{#1}}}
    \newcommand{\vTwo}[1]{\textcolor{teal!90!black}{\textbf{#1}}}
    \newcommand{\vThree}[1]{\textcolor{magenta!90!black}{\textbf{#1}}}
    \newcommand{\vFour}[1]{\textcolor{lime!90!black}{\textbf{#1}}}
    \newcommand{\vOneTwo}[1]{\textcolor{orange!90!black}{\textbf{#1}}}
    \newcommand{\vOther}[1]{\textcolor{black}{#1}}
    
    \newcommand{\bad}[1]{\textcolor{gray!50}{\sout{#1}}}

    \node[anchor=north west, draw=gray!40, rounded corners, fill=gray!5, font=\small\sffamily, inner sep=5pt,text width=5cm] at (-7.9, 15) {
        \textbf{Legend: Colors for $\bm \xi$}\\
        \vEmpty{blue}: $\bm \xi = (0,0,0,0)^\top$\\
        \vOne{red}: $\bm \xi = (1,0,0,0)^\top$\\
        \vTwo{green}: $\bm \xi =(0,1,0,0)^\top$\\
        \vThree{magenta}: $\bm \xi =(0,0,1,0)^\top$\\
        \vFour{lime}: $\bm \xi =(0,0,0,1)^\top$\\
        \vOneTwo{orange}: $\bm \xi =(1,1,0,0)^\top$\\
        \bad{grey}: $4\in \bm u$ and $\xi_{4}=0$ (redundant)
    };

    \node[box] (empty) at (-7,10) {
        $\bm u$\\[-2pt]
        \rule{1.5cm}{0.4pt}\\[2pt]
        $\bm \xi$
    };

    \node[box] (empty) at (0,0) {
        $ \varnothing$ \\[-2pt]
        \rule{1.5cm}{0.4pt}\\[2pt]
       \vEmpty{0000}
    };

    \node[box] (u1) at (-6, 3) {
        $\{1\}$ \\[-2pt] \rule{1.5cm}{0.4pt}\\[2pt]
        \vEmpty{0000} \\ \vOne{\fbox{1000}}
    };
    
    \node[box] (u2) at (-2, 3) {
        $\{2\}$ \\[-2pt] \rule{1.5cm}{0.4pt}\\[2pt]
        \vEmpty{0000} \\ \vTwo{\fbox{0100}}
    };
    
    \node[box] (u3) at (2, 3) {
        $\{3\}$ \\[-2pt] \rule{1.5cm}{0.4pt}\\[2pt]
        \vEmpty{0000} \\ \vThree{\fbox{0010}}
    };
    
    \node[box] (u4) at (6, 3) {
        $\{4\}$ \\[-2pt] \rule{1.5cm}{0.4pt}\\[2pt]
        \bad{0000} \\ 
        \vFour{\fbox{0001}}
    };

    
    \node[box] (u12) at (-6, 7) {
        $\{1, 2\}$ \\[-2pt] \rule{1.5cm}{0.4pt}\\[2pt]
        \begin{tabular}{c@{\hskip 0.1cm}c}
        \vEmpty{0000} & \vOne{1000} \\
        \vTwo{0100} & \vOneTwo{\fbox{1100}}
        \end{tabular}
    };

    \node[box] (u13) at (-3.6, 7) {
        $\{1, 3\}$ \\[-2pt] \rule{1.5cm}{0.4pt}\\[2pt]
        \begin{tabular}{c@{\hskip 0.1cm}c}
        \vEmpty{0000} & \vOne{1000} \\
        \vThree{0010} & \vOther{\fbox{1010}}
        \end{tabular}
    };

    \node[box] (u23) at (-1.2, 7) {
        $\{2, 3\}$ \\[-2pt] \rule{1.5cm}{0.4pt}\\[2pt]
        \begin{tabular}{c@{\hskip 0.1cm}c}
        \vEmpty{0000} & \vTwo{0100} \\
        \vThree{0010} & \vOther{\fbox{0110}}
        \end{tabular}
    };

    \node[box] (u14) at (1.2, 7) {
        $\{1, 4\}$ \\[-2pt] \rule{1.5cm}{0.4pt}\\[2pt]
         \begin{tabular}{c@{\hskip 0.1cm}c}
        \bad{0000} & \bad{1000} \\
        \vFour{0001} & \vOther{\fbox{1001}}
        \end{tabular}
    };

    \node[box] (u24) at (3.6, 7) {
        $\{2, 4\}$ \\[-2pt] \rule{1.5cm}{0.4pt}\\[2pt]
        \begin{tabular}{c@{\hskip 0.1cm}c}
        \bad{0000} & \bad{0100} \\
        \vFour{0001} & \vOther{\fbox{0101}}
        \end{tabular}
    };

    \node[box] (u34) at (6, 7) {
        $\{3, 4\}$ \\[-2pt] \rule{1.5cm}{0.4pt}\\[2pt]
        \begin{tabular}{c@{\hskip 0.1cm}c}
        \bad{0000} & \bad{0010} \\
        \vFour{0001} & \vOther{\fbox{0011}}
        \end{tabular}
    };

    \node[box, minimum width=4.5cm, minimum height=3.5cm] (u1234) at (0, 12) {
        $\{1, 2, 3, 4\}$ \\[-2pt] 
        \rule{4cm}{0.4pt}\\[3pt]
        \renewcommand{\arraystretch}{1.1}
        \begin{tabular}{c@{\hskip 0.5cm}c}
            \vOther{\fbox{1111}} & \bad{1110} \\
            \vOther{0111} & \bad{0110} \\
            \vOther{1011} & \bad{1010} \\
            \vOther{1101}  & \bad{1100} \\
            \vOther{0011} & \bad{0010} \\
            \vOther{0101} & \bad{0100} \\
            \vOther{1001} & \bad{1000} \\
            \vFour{0001} & \bad{0000} 
        \end{tabular}
    };

    \foreach \x in {u1, u2, u3, u4} \draw[link] (empty) -- (\x);
    
    \draw[link] (u1) -- (u12); \draw[link] (u1) -- (u13); \draw[link] (u1) -- (u14);
    \draw[link] (u2) -- (u12); \draw[link] (u2) -- (u23); \draw[link] (u2) -- (u24);
    \draw[link] (u3) -- (u13); \draw[link] (u3) -- (u23); \draw[link] (u3) -- (u34);
    \draw[link] (u4) -- (u14); \draw[link] (u4) -- (u24); \draw[link] (u4) -- (u34);

    \foreach \x in {u12, u13, u23, u14, u24, u34} \draw[link] (\x) -- (u1234);

\end{tikzpicture}
\caption{Visualization of all possible indices $\bm u$ and $\bm \xi$ of the terms $f_{\bm u,\bm \xi}$ for the case $d=3$. As calculated in~\eqref{eq:number_spherical_terms}, we have in total $49$ terms. In Definition~\ref{def:spherical_anova_final} we defined these terms iteratively. We illustrate this definition by using different colors for the terms belonging to different parity vectors $\bm \xi$. In Section~\ref{sec:redundancy} we explain which terms have to be omitted due to redundancies on the sphere. We have grayed out the corresponding vectors $\bm \xi$. This reduces the decomposition to 34 terms, as calculated in~\eqref{eq:number_spherical_terms2}. We mark all terms where we do not have the challenges in Section~\ref{sec:challenge} with a \fbox{box}. } 
\label{fig:diagram}
\end{figure}

\begin{example}\label{ex:x1}
Let the function be $f(\x) = x_1^2$, which is a pure even function, which depends only on $x_1$, but due to the fact that $\sum_{i=1}^d x_i^2 = 1$ on the sphere, the operators $P_{\bm u}$ applied to this function give
\begin{small}
\begin{align*}
P_{\varnothing}f &= \tfrac{1}{d+1}\\
P_{1}f &= x_1^2,\\
P_{i}f &=\frac{1}{d} (1-x_i^2)& \quad i =\{2,\ldots, d+1\},\\
P_{\{1\}\cup \bm a}f &=x_1^2 & \quad \text{ for }\bm a \subset [d+1]\backslash\{ 1\} ,\\
P_{\{\bm a\}}f &= \frac{1}{\omega_{d-|\bm a|}} (1-\norm{x_{\bm a}^2})\int_{\S^{|\bm a^c|-1}} x_1^2 \d \mu_{|\bm a^c|-1}(\x_{\bm a^c}) =  \frac{1}{d+1-|\bm a|}\left(1-\norm{\x_{\bm a}}^2\right)  & \quad \text{ for } \bm a \subset [d+1]\backslash\{ 1\}.
\end{align*}
\end{small}
For details about the calculation of the involved integrals, see also Appendix~\ref{sec:app1}.
Despite the function depending only on $x_1$, the application of the operator $P_{\bm u}$ on the function leads to non-constant terms, even if $1\notin \bm u$. In this easy example, we would expect that the spherical ANOVA decomposition is $f_{\varnothing,\bm 0} = \frac{1}{d+1} $, $f_{\{1\},\bm 0} = x_1^2 - \frac{1}{d+1} $ and all other terms are zero. The Definition~\ref{def:spherical_anova_final} with the operator $P_{\bm u}$ leads to a much more complicated decomposition than the expected one, it introduces unnecessarily additional terms. This is due to the natural dependence of the terms $f_{\bm u,\bm \xi}$ if some entry of $\bm \xi_{\bm u}$ is zero. Namely, Definition~\ref{def:spherical_anova_final} leads to
\begin{align*}
f_{\varnothing, \bm 0} &= \tfrac{1}{d+1},  & f_{\{1\}, \bm 0} &= x_1^2 - \tfrac{1}{d+1},\\
f_{\{i\}, \bm 0} &= \frac{1}{d} (1-x_i^2) - \tfrac{1}{d+1}, & f_{\{1,i\}, \bm 0} &=P_{\{1,i\}}f- f_{\{1\}, \bm 0} - f_{\{i\}, \bm 0} - f_{\varnothing, \bm 0} = - f_{\{i\}, \bm 0}.
\end{align*}
And also similar for the higher order terms, which are not zero, namely 
$$f_{\{i\}\cup \bm a,\bm 0} = - f_{\bm a,\bm 0} \quad \text{ for all }\bm a \subset \{2,3,\ldots, d+1\}.$$
Our desired function $f = x_1^2$ is completely included in two terms, more precisely, $f= f_{\varnothing,\bm 0} + f_{\{1\},\bm 0} $. The only non-zero terms are $f_{\bm a,\bm\xi}$ with $\bm a\subseteq \bm u =  \{1\}$. All the other terms cancel each other out, such that we can set them to zero. 
\end{example}
This example shows how we will modify Definition~\ref{def:spherical_anova_final}.  
To mitigate the problem of dependencies and make our spherical ANOVA decomposition more similar to the classical decomposition, we show the following.
\begin{lemma}\label{lem:0=sum}
Let the function $f(\x)$ have the form $f(\x)= g(\x_{\bm u})$, depend on all variables in $\bm u$ and have parity $\bm \xi$, i.e. $\Xi_{\bm \xi} f = f$. Let furthermore the parity vector $\bm \xi $ have the form $\bm \xi = (\bm 1_{\bm v},\bm 0_{\bm v^c})^\top$ with $\bm v\subset \bm u$.  
Then for all $\varnothing\neq \bm b\subseteq [d+1]\backslash \bm u$ we have for the spherical ANOVA terms defined in Definition~\ref{def:spherical_anova_final} that
\begin{equation}\label{eq:sum_terms=0}
0 = \sum_{\bm v \subseteq \bm a \subseteq \bm u}f_{\bm a\cup \bm b,\bm \xi } .
\end{equation}
\end{lemma}
\begin{figure}[b]
    \centering
    \begin{tikzpicture}[scale=0.6, transform shape,
        outline/.style={thick, draw=black},
        set label/.style={font=\Large\bfseries, text=black}]
        
        \draw[outline, fill=gray!10] (0,0) ellipse (3.5cm and 2.5cm);
        \node[set label] at (2, 1.5) {$\bm u$};
        
        \draw[outline, fill=gray!30] (-1.5,0) circle (1.2cm);
       \node[set label, align=center] at (-1.5, 0) {$\bm v$, \\ $\bm \xi_{\bm v} = \bm 1$};
        
        \draw[outline, fill=gray!30] (5.5,0) circle (1.2cm);
        \node[set label] at (5.5, 0) {$\bm b$};
    \end{tikzpicture}
    \caption{Illustration of the indices $\bm u, \bm v,\bm b \subseteq [d+1]$ in the proof of Lemma~\ref{lem:0=sum}.} 
    \label{fig:indices}
\end{figure}
\begin{proof}
We illustrate the indices in this proof in Figure~\ref{fig:indices}. 

We do the proof by induction over the order $|\bm b|$. First, let $\bm b =i$. Then by definition, 
\begin{align*}
f_{\bm u\cup  i,\bm \xi } &= P_{\bm u\cup i } [\Xi_{\bm \xi} f] - \sum_{\bm v\subseteq \bm a \subset \bm u\cup i} f_{\bm a,\bm \xi } = P_{\bm u } [\Xi_{\bm \xi} f] -   \sum_{\bm v\subseteq \bm a \subseteq \bm u} f_{\bm a,\bm \xi } - \sum_{\bm v\subseteq \bm a \subset \bm u} f_{\bm a \cup i,\bm \xi }\\
&= - \sum_{\bm v\subseteq \bm a \subset \bm u} f_{\bm a\cup i,\bm \xi }.
\end{align*}
Thereby we used that 
\begin{equation}\label{eq:f=sum_subsets}
f= \sum_{\bm v \subseteq\bm a\subseteq \bm u}f_{\bm a,\bm \xi}.
\end{equation}
It follows~\eqref{eq:sum_terms=0} for $\bm b =\{i\}$. For the induction step assume that~\eqref{eq:sum_terms=0} holds for all $\bm b $ with $|\bm b|< k$ and choose now $\bm b$ with $|\bm b|=k$, then
\begin{align*}
f_{\bm u\cup  \bm b,\bm \xi } &= P_{\bm u\cup \bm b } [\Xi_{\bm \xi} f] - \sum_{\bm v\subseteq \bm a \subset \bm u\cup \bm b} f_{\bm a,\bm \xi } \\
&= P_{\bm u } [\Xi_{\bm \xi} f] -   \sum_{\bm v\subseteq \bm a \subseteq \bm u} f_{\bm a,\bm \xi }- \sum_{\varnothing\neq \tilde{\bm b} \subset \bm b}\sum_{\bm v\subseteq \bm a \subseteq \bm u} f_{\bm a\cup \tilde{\bm b},\bm \xi } -\sum_{\bm v\subseteq \bm a \subset \bm u} f_{\bm a\cup \bm b,\bm \xi }\\
& =\underbrace{ f -   \sum_{\bm v\subseteq \bm a \subseteq \bm u} f_{\bm a,\bm \xi }}_{=0, \text{ due to \eqref{eq:f=sum_subsets}}} - \sum_{\varnothing\neq \tilde{\bm b} \subseteq \bm b}\underbrace{\sum_{\bm v\subseteq \bm a \subset \bm u} f_{\bm a\cup \tilde{\bm b},\bm \xi }}_{=0 \text{ due to induction assumption}} - \sum_{\bm v\subseteq \bm a \subset \bm u} f_{\bm a\cup \bm b,\bm \xi }.
\end{align*}
Hence,~\eqref{eq:sum_terms=0} is also fulfilled for $\bm b$ with $|\bm u|=k$.
\end{proof}
\changed{Finally, we state a definition for the final spherical ANOVA decomposition.}
\changed{
\begin{definition}\label{def:final_def}
To define the spherical ANOVA decomposition do the following steps:
\begin{enumerate}
    \item Construct the parity decomposition $f=\sum_{\bm \xi \in [0,1]^d} \Xi_{\bm \xi} f$, described in~\eqref{eq:def_even_odd_d}.
    \item Construct the term $f_{\bm u,\bm \xi}$ with Definition~\ref{def:spherical_anova_final}. 
    \item For every $\bm \xi \in [0,1]^{d+1}$: If  $\Xi_{\bm \xi} f$ does not depend on the variables $\tilde{ \bm u}^c$, i.e., does depend on all variables in $\tilde{ \bm u}$, then set
    $$f_{\bm u,\bm \xi} = 0 \quad \text{ for all }\quad \bm u \not\subseteq \tilde{\bm u}.$$
\end{enumerate}
\end{definition}}%
We will show in the following how to compute the spherical ANOVA decomposition numerically.

\section{Choice of basis}\label{sec:basis}
For periodic functions one can write every function in $L_2(\T^d)$ with the Fourier series. There is a direct connection between the index $\bm k$ of the Fourier coefficients and the ANOVA terms $f_{\bm u}$, i.e.\,which parts of the Fourier series belong to which ANOVA term, see~\cite{PoSc19a}. Similarly, for orthogonal polynomials with tensor product structure, there is an easy connection between the ANOVA index $\bm u$ and the frequency indices. See also~\cite{diss} for an overview about different basis choices for modeling the ANOVA decomposition on tensor product domains.\par
In our case here, with the sphere as domain, we do not have a tensor product structure of the domain available. However, for every spherical ANOVA term $f_{\bm u,\bm \xi}$ we construct in the following a basis~$\mathcal B_{\bm u,\bm \xi}$. 
\subsection{The one-dimensional terms}
Let $i \in [d+1]$. To describe the space of one-dimensional spherical ANOVA terms $f_{\{i\},\bm \xi}$ with a basis $\mathcal B_{\{i\},\bm \xi}$, we may choose an orthogonal basis $\eta_{k}(x_i)$, such that 
\begin{equation*}
\delta_{k,\ell} = \langle \eta_k(x_i),\eta_{\ell}(x_i)\rangle_{\S^d} = \int_{-1}^1 \eta_k(x_i)\,\eta_{\ell}(x_i)\, m_i(x_i)\d x_i,
\end{equation*}
which means that the basis functions have to be orthogonal with respect to $m_i(x_i) \changed{=\omega_{d-1} \left(1-x_i^2\right)^{(d-2)/2}}$ on $[-1,1]$, which are the \textbf{Gegenbauer polynomials} with parameter $\alpha = \tfrac{d-1}{2}$, usually denoted by $C_k^\alpha$, see~\cite{Ma64}. For $k\in \N$ define 
\begin{equation*}
C_{k}^\alpha(x) \coloneqq \frac{1}{\Gamma(\alpha)} \sum_{i = 0}^{\lfloor \tfrac k2 \rfloor}(-1)^i\frac{\Gamma(\alpha + k -i)}{i!\, (k-i)!} (2x)^{k-2i}.
\end{equation*}
The first polynomials have the form: 
\begin{align*}
C_{0}^{(\alpha)}(x) &= 1,\\
C_1^{(\alpha)} (x)&=2\alpha x,\\ 
C_2^{(\alpha)} (x)&=-\alpha +2\alpha(1+\alpha)x^2.
\end{align*}
For one-dimensional terms, there are only the two possibilities: even and odd functions. The Gegenbauer polynomials have the nice property that for odd $k$ they are odd functions and for even $k$ they are even functions. Since the linear spaces of odd and even function on $[-1,1]$ are orthogonal with respect to $m_i(x_i) \changed{=\omega_{d-1} \left(1-x_i^2\right)^{(d-2)/2}}$, we receive an orthogonal basis of the space of odd functions on $[-1,1]$ for describing the one-dimensional odd spherical ANOVA terms in Definition~\ref{def:spherical_anova_final} by
\begin{equation*}
   \mathcal B_{i,\bm e_i} =  \left\{C_{k}^{(d-1)/2}(x)\mid x \in [-1,1] \right\}_{k=1,3,5,\ldots},
\end{equation*}
which depend on the ambient dimension $d$. For the even one-dimensional term, naturally a basis can be constructed by the even Gegenbauer polynomials. Let $\bm u=\{i\}$, then due to our Definition~\ref{def:spherical_anova_final},
\begin{align*}
P_{\varnothing}f_{\{i\},\bm \xi}&= 
    P_{\varnothing }\E_{1}\AA_{\{1\}} f -  P_{\varnothing }P_{\varnothing } f  
=0,
\end{align*}
which means that the coefficient for the constant in the even one-dimensional term is zero, such that an even basis for $f_{\{i\},\bm 0}$ reads
\begin{equation*}
   \mathcal B_{i,\bm 0} =  \left\{C_{k}^{(d-1)/2}(x) \mid x \in [-1,1]\right\}_{k=2,4,6,\ldots}.
\end{equation*}

 \subsection{The higher-dimensional terms}

 For describing the spherical ANOVA terms $f_{\bm u,\bm \xi}$ on the domain $\B^{|\bm u|}$, we are searching for an orthogonal basis $\mathcal B_{\bm u,\bm \xi}=\{\eta_k(\x_{\bm u})\mid\x_{\bm u} \in\B^{|\bm u|}\}_{k}$ with orthogonality with respect to the weight function $m_{\bm u}(\x_{\bm u})\changed{=\omega_{d-|\bm u|}\left(1-\norm{\x_{\bm u}}^2\right)^{(d-|\bm u|)/2}}$, since then the basis is orthogonal on $\S^d$ with respect to the uniform measure on the sphere $\mu_d$. Changing coordinates to polar coordinates, this weight function becomes
 $$m_{\bm u}(\x_{\bm u})\changed{=\omega_{d-|\bm u|}  \left(1-r^2\right)^{(d-|\bm u|)/2}},$$
 which leads us to the Jacobi polynomials.

\begin{remark}
\textbf{Jacobi polynomials} $P^{\alpha,\beta}_n(x)$ form an orthonormal basis in $L_2( [-1,1])$ with respect to the weight $(1-x)^\alpha \, (1+x)^\beta$. 
For polynomials on $[0,1]$ we use the substitution $t=r^2-1$, which means that
\begin{align*}
    \delta_{k,\ell}&= \int_{-1}^1 P^{\alpha,\beta}_k(t) P^{\alpha,\beta}_{\ell}(t) (1-t)^\alpha (1+t)^\beta\d t\\
    &=4 \int_{0}^1 P^{\alpha,\beta}_k(2r^2-1) P^{\alpha,\beta}_{\ell}(2r^2-1)\,  [2(1-r^2)]^{\alpha}(2r^2)^{\beta} r \d r\\
    &=4 \,\cdot  2^{\alpha +\beta} \int_{0}^1 P^{\alpha,\beta}_k(2r^2-1) P^{\alpha,\beta}_{\ell}(2r^2-1)\,  (1-r^2)^{\alpha}\,r^{2\beta+1}  \d r.
\end{align*}
\end{remark}

With this, an orthogonal basis on $\B^{n}$, used in the literature~\cite[Prop. 11.1.13]{DaXu13}, \cite[Prop. 5.2.1]{DuXu14} is the following.
For $N\in \N$, $0\leq j\leq \tfrac N2$ and $1\leq k \leq \dim (\mathcal{H}_{N-2j}(\mathbb{S}^{|\bm u|}))$ define
\begin{equation}\label{eq:PjkN}
P_{j,k}^N(\x ) = P_{j}^{\tfrac{d-1-|\bm u|}{2},N-2j+\frac{d-2}{2}}(2\norm{\x}^2-1)Y_{k,{N-2j}}(\x).
\end{equation}
These form an orthogonal basis of the space of orthogonal polynomials of degree exact $N$ with respect to the inner product 
$$\langle f(\x_{\bm u}),g(\x_{\bm u})\rangle_{\S^d} = \langle f(\x_{\bm u}),g(\x_{\bm u})\rangle_{(\B^{|\bm u|}, m_{\bm u})}= c \int_{\B^{|\bm u|}} f(\x_{\bm u}) g(\x_{\bm u}) m_{\bm u}(\x_{\bm u})\dx\x_{\bm u},$$
where $c$ is a normalization constant. 

\subsubsection{The two-dimensional terms}
In the two-dimensional case the spherical harmonics are precisely the cosine and sine functions. For the sake of simplicity of the notation, we will construct here a basis for the terms $f_{\{1,2\},\bm \xi}$, but for all other two-dimensional terms the procedure is the same.
\par
In two-dimensional polar coordinates $(x_1,x_2)=(r\cos\theta,r\sin\theta)$,
as basis $Y_{k,{N-2j}}(\x)$ we use~\eqref{eq:Y_for_d=2}. Then, the orthogonal basis in~\eqref{eq:PjkN} for $\B^2$ reads
\begin{small}
\begin{align}\label{eq:PjkN2}
&\left\{P_{j}^{\tfrac{d-3}{2},N-2j}(2r^2-1)r^{N-2j}\cos((N-2j)\theta) \right\}_{N\in \N_0, 0\leq j \leq \tfrac N2}\notag\\
&\quad \bigcup \left\{P_{j}^{\tfrac{d-3}{2},N-2j}(2r^2-1)r^{N-2j}\sin((N-2j)\theta)\right\}_{N\in \N_0, 0\leq j < \tfrac N2}.
\end{align}
\end{small}
In polar coordinates, the parity vector $\bm \xi =(\xi_1,\xi_2,\bm 0_{d-1})^\top$ defines the parity by
$$g(r,\theta)=(-1)^{\xi_1}g(r,-\theta), \quad g(r,\theta)=(-1)^{\xi_2} g(r,\pi-\theta). $$

In~\eqref{eq:PjkN2} we have to classify the basis functions to the parity vectors and furthermore, due to the dependencies of the spherical ANOVA terms we have to be careful additionally and have to omit some of the basis functions, which are already included in the terms of lower order $0$ or $1$. \par

The selection rule is based on the polynomial degree $N$ and the resulting parity vector $\bm \xi_{\{1,2\}} \in \{0,1\}^2$ of the term $f_{\bm u,\bm\xi}$. The term $N-2j$ has the same parity as $N$, such that a distinction between even and odd $N$ is useful.
We exclude all constant ($N=0$) and linear ($N=1$) terms. For the purely even parity $\bm \xi = (0,0)$, we additionally exclude quadratic terms ($N=2$) that decompose into sums of univariate functions.
The parity is determined by $N$ and the angular function type, as detailed in Table~\ref{tab:basis_selection}. There the basis functions for constructing $\mathcal B_{\{1,2\},\bm \xi}$ are summarized.

\begin{table}[htb]
    \centering
    \renewcommand{\arraystretch}{1.2}
    \begin{tabular}{@{}ccclcll@{}}
        \toprule
        $N$ & $j$ &\textbf{Type} & \textbf{Form}  & $\bm \xi_{\{1,2\}}$ & \textbf{Action} & \textbf{Reason} \\
        \midrule
        0 &0& $\cos$ & $1$ & $(0,0)$ & \textcolor{red!60!black}{Omit} & Covered by $f_{\varnothing,\bm 0}$ \\
        \addlinespace
        1 &0& $\cos$ & $x_{1}$ & $(0,1)$ & \textcolor{red!60!black}{Omit} & Covered by $f_{\{1,{\bm e_1}\}}$ \\
        1 &0& $\sin$ & $x_{2}$ & $(1,0)$ & \textcolor{red!60!black}{Omit} & Covered by $f_{\{2,{\bm e_2}\}}$ \\
        \addlinespace
        2 &0& $\cos$ & $x_{1}^2 - x_{2}^2$ & $(0,0)$ & \textcolor{red!60!black}{Omit} & Sum of 1d terms \\
        2 &1& $\cos$ & $x_{1}^2 + x_{2}^2$ & $(0,0)$ & \textcolor{red!60!black}{Omit} & Sum of 1d terms \\
        2 &0& $\sin$ & $x_{1} x_{2}$ & $(1,1)$ & \textcolor{green!40!black}{\textbf{Keep}} & True interaction \\
        \addlinespace
        $\geq 3$, odd &$0\leq j\leq \tfrac{ N-3}{2}$& $\cos$ & & $(1,0)$  & \textcolor{green!40!black}{\textbf{Keep}} & \\
        $\geq 3$, odd &$j=\tfrac{ N-1}{2}$& $\cos$& & $(1,0)$  & \textcolor{red!60!black}{Omit} & Redundancy\\
        $\geq 3$, odd &$0\leq j\leq \tfrac{ N-3}{2}$& $\sin$ & & $(0,1)$  & \textcolor{green!40!black}{\textbf{Keep}} & \\
        $\geq 3$, odd &$j=\tfrac{ N-1}{2}$& $\sin$& & $(0,1)$  & \textcolor{red!60!black}{Omit} &Redundancy \\
         $\geq 4$, even &$0\leq j\leq \tfrac N2-2$& $\cos$ & & $(0,0)$ & \textcolor{green!40!black}{\textbf{Keep}} & \\
              $\geq 4$, even &$\tfrac N2-1\leq j\leq \tfrac N2$& $\cos$ & & $(0,0)$ & \textcolor{red!60!black}{Omit} & Redundancy\\
        $\geq 4$, even &$0\leq j< \tfrac N2$& $\sin$ &  & $(1,1)$ & \textcolor{green!40!black}{\textbf{Keep}} & \\
        \bottomrule
    \end{tabular}

      \caption{Classification and selection of two-dimensional basis functions based on degree $N$ and angular frequency $m=N-2j$. Terms marked as \textit{Omit} are redundant with lower-order ANOVA terms.}
    \label{tab:basis_selection}
\end{table}

    \textbf{The case $N=2$}:\\
   A rationale for discarding the terms with $N=2$ for $\bm \xi = (0,0)$ is the specific geometry of the sphere. The radial term $P_1(2r^2-1)$ and the angular term $r^2 \cos(2\theta)$ correspond to linear combinations of $x_{1}^2$ and $x_{2}^2$. Since the even one-dimensional terms $f_{\{i\}, \bm 0}$ already span the space of $x_{i}^2$ (orthogonalized against the constant), including these $N=2$ terms in the interaction basis would introduce linear dependencies.
    In contrast, the term $r^2 \sin(2\theta) \sim x_{1} x_{2}$ is kept as it belongs to the parity class $\bm \xi_{\{1,2\}} = (1,1)$ and represents a pure interaction that cannot be formed by sums of univariate functions.\par
\bigskip
    
    For $N\geq 3$ only the parity $\bm \xi_{\{1,2\}}=(1,1)$ is carefree. Concerning the other cases, we have additional redundancy, which leads to more omitted terms. The reason for this is as follows.\par
    \medskip
    \textbf{The case $\bm \xi_{\{1,2\}}=(0,0)$}:\\
    Let $N$ be even, then the space of homogeneous polynomials in two variables $x_1$ and $x_2$ of the form $x_1^{2k}x_2^{2\ell}$ with $2k+2\ell = N$ has dimension
    $\tfrac N2 +1$, which includes the two one-dimensional terms $x_1^N$ and $x_2^N$. This means from the possibilities for $j$, namely $0\leq j\leq \tfrac N2$ we can keep only $\tfrac N2 -1$ and have to omit two, which are redundant. \par
    \medskip
       \textbf{The case $\bm \xi_{\{1,2\}}=(1,0)$}: \\
Let $N$ be odd, then the space of homogeneous polynomials in two variables $x_1$ and $x_2$ of the form $x_1^{2k+1}x_2^{2\ell}$ with $2k+1+2\ell = N$ has dimension
    $\tfrac{N+1}{2} $, which includes the two one-dimensional term $x_1^N$. This means from the possibilities for $j$, namely $0\leq j\leq \tfrac{N-1}{2}$ we can keep only $\tfrac{N-1}2 $ and have to omit one, which is redundant. The case $\bm \xi_{\{1,2\}}=(0,1)$ is similar.

A similar approach can be used for higher-dimensional terms. However, we will limit ourselves here to this explanation for two-dimensional terms. We do reduce the redundancy between the terms of different order. However these terms are not necessarily orthogonal.

\section{Low dimensional approximation}\label{sec:low_dim}
In many applications, functions on the sphere depend on a large number of variables but are effectively dominated by terms of low order. The ANOVA decomposition allows us to identify this effective dimension and approximate the function by a sum of terms involving only a few variables. This is particularly valuable for mitigating the \textbf{curse of dimensionality} and for developing efficient quadrature or interpolation algorithms. Approaches for low dimensional approximation without the ANOVA decomposition can be found for example in~\cite{Xie22}.\par

Similar to the approaches~\cite{PoSc19a, LiPoUl23} for the classical ANOVA decomposition or~\cite{PoWe24} for the generalized ANOVA decomposition, we propose to approximate the function $f$ using only low-dimensional terms, i.e.,
$$f\approx f_{\varnothing} + \sum_{1 \le |\bm u|\leq q}\sum_{\bm \xi\in \U_{\bm u}}f_{\bm u,\bm \xi}\quad \text{ for some } q\ll d.$$
The next step is to approximate every term by using the specific basis $\mathcal B_{\bm u,\bm\xi}=\{\eta_k^{\bm u, \bm \xi}\}_{k\in \I_{\bm u}}$ introduced in Section~\ref{sec:basis},
\begin{equation}\label{eq:f=sum_a_eta}
f_{\bm u,\bm \xi}(\x_{\bm u}) \approx \sum_{k\in \I_{\bm u}} a^{\bm u,\bm \xi}_{k}\,\eta_k^{\bm u, \bm \xi}(\x_{\bm u}).
\end{equation}
Substituting the basis expansion~\eqref{eq:f=sum_a_eta} into the truncated ANOVA decomposition leads to the global approximation ansatz
\begin{equation*}
    f(\x) \approx a_{\varnothing} + \sum_{1 \le |\bm u|\leq q} \sum_{\bm \xi\in \U_{\bm u}} \sum_{k\in \I_{\bm u}} a^{\bm u,\bm \xi}_{k}\,\eta_k^{\bm u, \bm \xi}(\x_{\bm u}).
\end{equation*}
To determine the unknown coefficients $a^{\bm u,\bm \xi}_{k}$, we rely on a set of sampling nodes 
$$\mathcal X = \{\x^{(1)}, \dots, \x^{(M)}\} \subset \S^d$$ and the corresponding function values $\bm f = (f(\x^{(1)}), \dots, f(\x^{(M)}))^\top \in \R^M$.
Let $N$ denote the total number of basis functions in the truncated expansion and let $\bm a \in \R^N$ be the vector stacking all coefficients $a^{\bm u,\bm \xi}_{k}$.
Evaluating the approximation ansatz at the sampling nodes yields the linear system of equations
\begin{equation*}
    \bm A \bm a \approx \bm y,
\end{equation*}
where the system matrix $\bm A \in \R^{M \times N}$ contains the evaluations of the basis functions. Specifically, the entry $A_{i, \ell}$ corresponds to the $\ell$-th basis function evaluated at the $i$-th sampling node $\x^{(i)}$.\par

In practical applications, the number of samples $M$ is typically chosen to be larger than the number of unknowns $N$, resulting in an overdetermined system. Consequently, we seek the coefficient vector $\bm a$ that minimizes the residual in the Euclidean norm, solving the least squares problem
\begin{equation}\label{eq:least_squares}
    \min_{\bm c \in \R^N} \|\bm A \bm c - \bm y\|_2.
\end{equation}
While this problem can be solved using direct methods (e.g., via QR decomposition or SVD), the dimensions of $\bm A$ can become large for high-dimensional approximations. Moreover, as discussed in previous sections, the basis on the sphere may exhibit linear dependencies or ill-conditioning for certain parity combinations.
To address these challenges efficiently, we propose to solve~\eqref{eq:least_squares} using the \textbf{LSQR algorithm}~\cite{PaSa82}. LSQR is an iterative Krylov subspace method that is particularly well-suited for large, sparse, or ill-conditioned systems and possesses intrinsic regularization properties when the iteration is terminated early.\par

Due to the dependence of the terms of different order, the following procedure may be beneficial: We propose to do the approximation first using all terms of order one. Then, we do a second approximation for the residuum using all terms of order $2$, which could be done further.  
In Section~\ref{sec:numerics} and Figure~\ref{fig:sobol_indices} we apply both possibilities, the plain LSQR with all terms at once and the iterative increase of the order of the terms. 

\subsection{Variances for sensitivity analysis}
Global sensitivity analysis has become a cornerstone of computer experiments, providing essential insights into the input-output dependencies of numerical simulation models. Building on the foundational works~\cite{So90, So95}, the field has evolved significantly over the last two decades, yielding a wide array of advanced statistical methodologies and refinements of the original Sobol’ indices.\par
While classical variance-based sensitivity analysis (Sobol' indices) relies on the orthogonality of the ANOVA terms to decompose the total variance into a sum of partial variances, this property assumes independent input variables. On the sphere $\S^d$, the inputs are intrinsically coupled via the constraint $\norm{\x}=1$. Consequently, for even-parity terms, a strict orthogonality cannot be maintained, and the covariance between interaction terms is generally non-zero. 
Despite this lack of perfect orthogonality, the variances $\sigma^2(f_{\bm u,\bm \xi})$ remain meaningful sensitivity measures for the following reasons:
\begin{itemize}
    \item Following the framework of sensitivity analysis for dependent variables, see for example~\cite{KuTaAn12, PoWe24}, the contribution of a variable is always a combination of its marginal effect and its correlation with other variables. On the sphere, these correlations are not artifacts but physical properties of the domain. Our indices capture the combined effect of a variable and its induced changes on the sphere.
    \item A crucial advantage of our proposed construction is the rigorous exclusion of redundant basis terms. By ensuring that our dictionary of basis functions is linearly independent, we prevent variance inflation. In a redundant system, variance could be arbitrarily shifted between linearly dependent terms (e.g., between $f_{1}$ and a redundant part of $f_{1,2}$), leading to unstable and uninterpretable indices. Our restrictive construction ensures that the computed variances are stable and reflect the intrinsic coupling of the sphere rather than numerical artifacts.
\end{itemize}

\begin{lemma}
The spherical ANOVA terms defined in Definition~\ref{def:spherical_anova_final} have mean zero, i.e.,
$$P_{\varnothing} f_{\bm u,\bm \xi} = 0 \quad \text{ for all } \bm u\in \D\backslash\varnothing \text{ and all } \bm \xi\in \U_{\bm u}. $$
\end{lemma}
\begin{proof}
Since the spherical ANOVA terms $f_{\bm u,\bm \xi}$ defined in Definition~\ref{def:spherical_anova_final} have parity $\bm \xi$, the mean is automatically zero, if the function $f_{\bm u,\bm \xi}$ is odd in at least one direction. The only remaining case is when $\bm \xi = \bm 0$. But in this case we have for $\bm u\neq \varnothing$,
\begin{align*}
    P_{\varnothing }f_{\bm u,\bm 0 } &= P_{\varnothing} \left( P_{\bm u} [\Xi_{\bm 0}f](\x_{\bm a})- \sum_{ \bm a \subsetneq \bm u} f_{\bm a, \bm 0}(\x_{\bm a})\right)
    = P_{\varnothing}f -  P_{\varnothing}\left(\sum_{ \bm a \subsetneq \bm u} f_{\bm a, \bm 0}(\x_{\bm a})\right)\\
    &=   P_{\varnothing}\left(\sum_{ \varnothing \neq \bm a \subsetneq \bm u} f_{\bm a, \bm 0}(\x_{\bm a})\right) = \sum_{ \varnothing \neq \bm a \subsetneq \bm u}P_{\varnothing} f_{\bm a, \bm 0}(\x_{\bm a}).
\end{align*}
By induction this finishes the proof by using the induction assumption.
\end{proof}
Using the previous lemma, the variances of the spherical ANOVA terms can be easily calculated by 
\begin{align*}
\sigma^2(f_{\bm u,\bm\xi}) &\coloneqq \int_{\S^d} |f_{\bm u,\bm\xi}(\x_{\bm u}) -\int_{\S^d}f_{\bm u,\bm\xi}(\x_{\bm u}) \d{\mu_d}(\x)|^2 \d{\mu_d}(\x)
=\int_{\S^d} |f_{\bm u,\bm\xi}(\x_{\bm u})|^2 \d{\mu_d}(\x)\\
   &\stackrel{\eqref{eq:subs_Bsp}}{=}
 \int_{\B^{|\bm u|}} |f_{\bm u,\bm\xi}(\x_{\bm u})|^2\, m_{\bm u}(\x_{\bm u})
   \d \x_{\bm u},
\end{align*}
which is a weighted variance of the term $f_{\bm u}$ on $\B^{|\bm u|}$. In Section~\ref{sec:basis} we introduced the basis $\mathcal B_{\bm u,\bm \xi}$, which is orthogonal with respect to the weight function $m_{\bm u}$. Using the approximation~\eqref{eq:f=sum_a_eta}, the variances are approximated by
\begin{equation*}
\sigma^2(f_{\bm u,\bm\xi}) \approx  \sum_{\x\in \mathcal X} \left|f_{\bm u,\bm\xi}(\x)\right|^2 \approx \sum_{\x\in \mathcal X} \left|\sum_{k\in \I_{\bm u}} a^{\bm u,\bm \xi}_{k}\eta_k(\x_{\bm u})\right|^2.
\end{equation*}
Since terms with different parity vector are orthogonal with respect to $m_{\bm u}$, we define for $\bm u\neq \varnothing$ the Sobol indices on the sphere,
\begin{equation}\label{eq:S_u}
S_{\bm u} \coloneqq \frac{1}{\sigma^2(f)}\sum_{\bm \xi \in \U_{\bm u}}\sigma^2(f_{\bm u,\bm\xi}) \approx \frac{1}{\sigma^2(\bm f)} \sum_{\bm \xi \in \U_{\bm u}} \sum_{\x\in \mathcal X} \left|\sum_{k\in \I_{\bm u}} a^{\bm u,\bm \xi}_{k}\eta_k(\x_{\bm u})\right|^2 \eqqcolon \tilde{S}_{\bm u},
\end{equation}
where $\sigma^2(\bm f)$ is the variance of the data vector $\bm f$ and the expression $\tilde{S}_{\bm u}$ can be calculated easily using parts of the the matrix $\bm A$ and the coefficient vector $\bm a$. 
We interpret the normalized ratio $S_{\bm u}$ and its approximation $\tilde{S}_{\bm u}$ not strictly as a percentage of disjoint variance, but as a relative importance index reflecting the stable amplitude of the corresponding spherical interaction terms. \par
For the classical ANOVA decomposition the Sobol indices (or sometimes called global sensitivity indices) also denote the fraction of the variance of the ANOVA term and the total variance of the function. Due to the orthogonality between all terms, they sum up to one.

\section{Examples and numerical tests}\label{sec:numerics}
In this chapter, we evaluate the practical applicability and accuracy of the derived spherical ANOVA decomposition through numerical experiments. The primary objective is to demonstrate that the proposed spherical ANOVA decomposition provides a meaningful and interpretable representation of functions on $\S^d$.\par

To this end, we consider a set of test functions with known interaction structures. We compute the variances of the individual ANOVA terms $S_{\bm u}$ from~\eqref{eq:S_u} numerically using the proposed basis approximation and the LSQR solver in Python. 

Let us consider the following test functions $f_{A}, \ldots, f_F  \colon \S^{10}\rightarrow\R$, defined by
\begin{allowdisplaybreaks}
\begin{align*}
f_A(\x) &\coloneqq  x_1  x_2^3 + 2 x_3  x_4^5 + 0.05x_5 \\
f_B(\x) &\coloneqq  x_2x_1^2 \\
f_C(\x) &\coloneqq  x_1^4 + x_2^2 \\
f_D(\x) &\coloneqq  5x_1 x_2^2+ x_4 + \mathrm e^{x_3}+  10 \sin(3 \,\pi x_5)\, x_2^4 \\
f_E(\x) &\coloneqq  \sin(x_1) + 7 \,\sin^2(x_2) + 0.1\,x_3^4 \,\sin(x_1) \\
f_F(\x) &\coloneqq x_1^2x_2^2x_3^2
\end{align*}
\end{allowdisplaybreaks}

The spherical ANOVA decomposition of some of these functions is calculated in details in Appendix~\ref{sec:app1}.
For the numerical tests we draw $M = 10000$ random samples on the sphere $\S^{10}$, which is relatively sparse in this dimension, and use the basis described in Section~\ref{sec:basis} of order up to $q=2$ and polynomials of maximal total order $N=10$. We use~\eqref{eq:S_u} to approximate the Sobol indices with this basis. In Figure~\ref{fig:sobol_indices} we illustrate the resulting indices.
\begin{figure}[p]
    \centering
    
    \begin{subfigure}[b]{0.48\linewidth}
        \centering
        \includegraphics[width=1.1\linewidth]{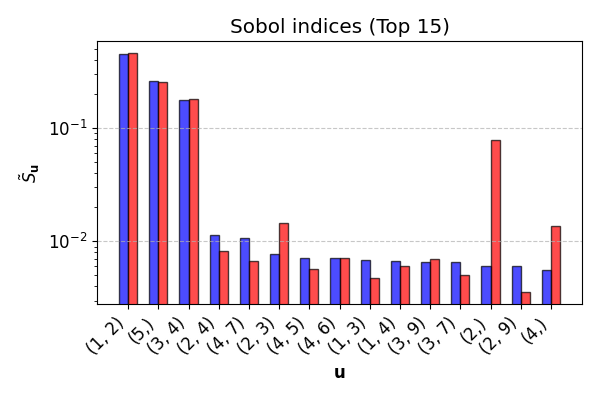}
        \caption{Test function $f_A$.}
        \label{fig:bild1}
    \end{subfigure}
    \hfill 
    \begin{subfigure}[b]{0.48\linewidth}
        \centering
        \includegraphics[width=1.1\linewidth]{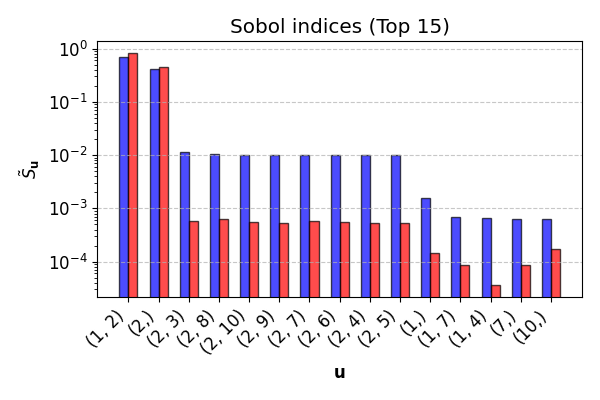}
        \caption{Test function $f_B$.}
        \label{fig:bild2}
    \end{subfigure}
    
    \vspace{1em} 
    
    \begin{subfigure}[b]{0.48\linewidth}
        \centering
        \includegraphics[width=1.1\linewidth]{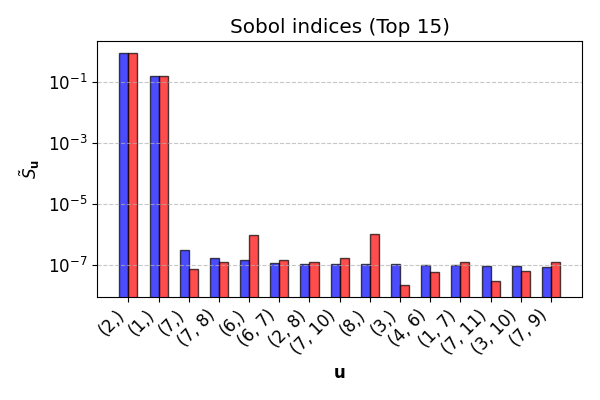}
        \caption{Test function $f_C$.}
        \label{fig:bild3}
    \end{subfigure}
    \hfill
    \begin{subfigure}[b]{0.48\linewidth}
        \centering
        \includegraphics[width=1.1\linewidth]{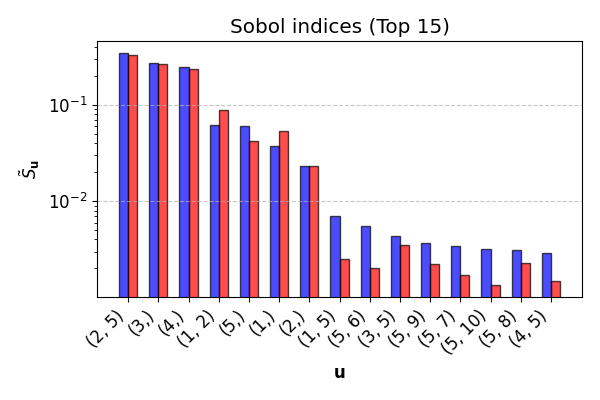}
        \caption{Test function $f_D$.}
        \label{fig:bild4}
    \end{subfigure}
    
    \vspace{1em} 
    
    \begin{subfigure}[b]{0.48\linewidth}
        \centering
        \includegraphics[width=1.1\linewidth]{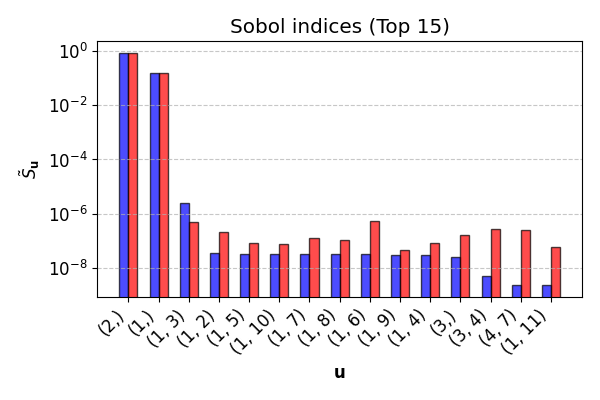}
        \caption{Test function $f_E$.}
        \label{fig:bild5}
    \end{subfigure}
    \hfill
    \begin{subfigure}[b]{0.48\linewidth}
        \centering
        \includegraphics[width=1.1\linewidth]{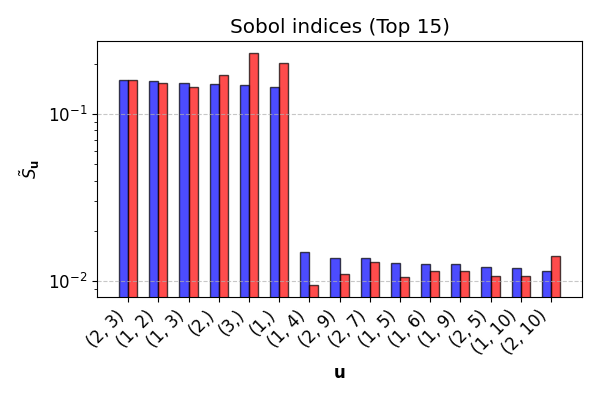}
        \caption{Test function $f_F$.}
        \label{fig:bild6}
    \end{subfigure}
    
    \caption{Approximated Sobol indices~\eqref{eq:S_u} for the different test functions. We plot the results from LSQR with all terms up to order $q=2$ together (red) and the procedure of increasing the order of the used terms iteratively (blue). Note that we use logarithmic scaling.}
    \label{fig:sobol_indices}
\end{figure}

\changed{The following results can be seen in the numerics.}
\begin{itemize}
    \item If all parts of the function are odd functions, like for $f_A$, the non-zero terms are detected easily. One can see in the diagram that it can happen that plain LSQR puts some variance in one-dimensional terms, which should be zero. 
    \item The function $f_B$ consists of one term with mixed parity. We clearly see here the effect, that the terms $f_{\{1,i\}}$ for $i>2$ are not zero, but in the order of magnitude of the Sobol indices $\tilde{S}_{\bm u}$ one can see that despite the dependence of the terms, our algorithm is able to detect the two important term $f_{\{1,2\},\bm 0}$ and $f_{\{2\},\bm 0}$.
    \item The function $f_{C}$ is a sum of two even functions. Our algorithm clearly detects these two one-dimensional terms. Due to the dependence of the even terms, we get some variance for the terms with $\bm u = \{2,i\}$ for $i\geq 3$, the reason was described in~\ref{sec:challenge}. Here, the order of magnitude of the misleading terms is slightly smaller for the case where we use LSQR with all terms of order $0,1$ and $2$ together.   
    \item If there are other terms than polynomials involved, like in the function $f_D$, we still are able to detect the seven non-zero terms in the spherical ANOVA decomposition.
    \item $f_E$ is the Ishigami function, where we are able to find the three non-zero terms. 
    \item The function $f_F$ is a term of order $3$, which means that our approximation with terms of order maximal $q=2$ could not model the function exactly. However, also here we see clearly the non-zero terms belonging to $\bm u\subset \{1,2,3\} $.
\end{itemize}

\subsection{Summary and outlook}
The spherical ANOVA decomposition is a powerful mathematical framework designed to analyze high-dimensional functions defined on the high-dimensional sphere $\S^d$. Unlike the classical ANOVA, this approach respects the intrinsic geometry of the sphere, making it essential for problems involving directional data or isotropic constraints.

Our decomposition allows for the breakdown of a complex function into a sum of orthogonal terms, where each term represents the contribution of specific subsets of variables. This reveals the effective dimension of the function.
\par
By calculating global to Sobol indices, we can quantify how much variance is driven by single variables versus complex interactions, effectively opening the black box of high-dimensional models.
 We are able to construct sparse approximation algorithms that avoid the curse of dimensionality by identifying that a function is dominated by low-order interactions. It paves the way for designing specialized cubature rules that focus computational effort only on the relevant subspaces identified by the ANOVA terms. In fields dealing with normalized feature vectors, our spherical ANOVA decomposition can serve as a robust feature selection and noise reduction tool.

\section*{Acknowledgments}
I would like to thank Daniel Potts for many helpful discussions. This work was funded by the Deutsche Forschungsgemeinschaft (DFG, German Research Foundation) - project number 569580074.

\bibliographystyle{abbrvurl}
\bibliography{references}

\begin{appendix}
\section{\changed{Useful definitions and results}}
In this \changed{appendix} we collect basic concepts and results which \changed{we are using} to construct a spherical ANOVA decomposition.
\subsection{Variable substitutions on the sphere \texorpdfstring{$\S^d$}{} and the ball \texorpdfstring{$\B^d$}{}}
We collect some substitutions here.
Splitting the integral over the ball into two integrals over the sphere and the radius by using polar coordinates leads to the following basic lemma. 
\begin{lemma}\label{lem:subs_ball}
For a function $g\colon \B^d \rightarrow \R$ the integral on the $d$-dimensional ball can be decomposed into
    \begin{equation*}
    \int_{\B^{d}} g(\x) \d\x = \int_{0}^1\int_{\S^{d-1}} g(r\z)\, r^{d-1}  \d \mu_{d-1}(\z) \d r. 
    \end{equation*}
\end{lemma}
\subsubsection{Variable substitutions on the sphere}
Let us decompose the surface measure $\mu_d$ by the following. For every $\bm x\in \S^d$ and $ \varnothing \neq \bm u \subset [d+1]$ we can write as $1=\norm{\x_{\bm u}}^2+ \norm{\x_{\bm u^c}}^2$. 
From \cite[Lemma A.5.4.]{DaXu13} and Lemma~\ref{lem:subs_ball} follows for $|\bm u|\geq 2$ and $\x_{\bm u}\in \S^{|\bm u|-1}$,
\begin{small}
 \begin{align}\label{eq:subs_B}
   &\int_{\S^{|\bm v| + |\bm u|-1}} g(\x) \d \mu_{|\bm v| + |\bm u|-1}(\x) = \int_{\B^{|\bm v|}} \changed{\left(1-\norm{\bm x_{\bm v}}^2\right)^{|\bm u|/2-1}}\int_{\S^{|\bm u|-1}} g\left(\x_{\bm v},\sqrt{1-\norm{\x_{\bm v}}^2} \x_{\bm u}\right)  \d \mu_{|\bm u|-1}(\x_{\bm u}) \d \x_{\bm v} \\
   &=\quad  \int_{0}^1 \int_{\S^{|\bm v|-1}} \changed{\left(1-r^2\right)^{|\bm u|/2-1}}r^{|\bm v|-1}\int_{\S^{|\bm u|-1}} g(r\z_{\bm v},\sqrt{1-r^2} \x_{\bm u})   \d \mu_{|\bm u|-1}(\x_{\bm u}) \d \mu_{|\bm v|-1}(\z_{\bm v}) \d r .\notag
    \end{align}
    \end{small}
    For the case where $\bm u= i$ we write for $\x_{\bm v\cup i}\in \S^{|\bm v|}$, $\x_{\bm v \cup i} = (\sqrt{1-r^2}\,\y_{\bm v},r)$, where $\x_{\bm v}\in \S^{|\bm v|-1}$ and $r\in [-1,1]$. It follows that
    $$\d \mu_{|\bm v|}(\x_{\bm v\cup i}) = \changed{\left(1-r^2\right)^{|\bm v|/2-1}} \d r \d \mu_{|\bm v|-1} (\x_{\bm v}).$$
    Changing the variables $\x_{\bm v \cup i} \rightarrow (\sqrt{1-r^2}\y_{\bm v},r)$ gives
 \begin{equation*}
   \int_{\S^{|\bm v| }} g(\x) \d \mu_{|\bm v|}(\x) 
   = \int_{-1}^1\int_{\S^{|\bm v|-1}}   g\left(\sqrt{1-r^2}\,\y_{\bm v},r  \right) \changed{\left(1-r^2\right)^{|\bm v|/2-1}}\d \mu_{|\bm v|-1}(\y_{\bm v})   \d r .
    \end{equation*}

We use \eqref{eq:subs_B} and obtain in the special case 
$\x=(\x_{\bm u^c},\x_{\bm u}) =(r \z_{\bm u^c},\x_{\bm u}) \in \S^d$ with $d=|{\bm u}|+|{\bm u^c}|-1$, $\z_{\bm u^c}\in \S^{|\bm u^c|-1}$ 
\begin{align}\label{eq:subs_Bsp}
   &\int_{\S^{d}} g(\x) \d \mu_{d}(\x) = 
   \int_{\B^{|\bm u^c|}} \changed{\left(1-\norm{\bm x_{\bm u^c}}^2\right)^{|\bm u|/2-1}}
   \int_{\S^{|\bm u|-1}} g\left(\x_{\bm u^c},\sqrt{1-\norm{\x_{\bm u^c}}^2} \,\x_{\bm u}\right)  \d \mu_{|\bm u|-1}(\x_{\bm u}) \d \x_{\bm u^c} \\
   &=\quad  \int_{0}^1 \int_{\S^{|\bm u^c|-1}} \changed{\left(1-r^2\right)^{|\bm u|/2-1}}r^{|\bm u^c|-1}\int_{\S^{|\bm u|-1}} g(r\z_{\bm u^c},\sqrt{1-r^2}\, \x_{\bm u})   \d \mu_{|\bm u|-1}(\x_{\bm u}) \d \mu_{|\bm u^c|-1}(\z_{\bm u^c}) \d r .\notag 
    \end{align}

\subsubsection{Integrals on spheres with radius \texorpdfstring{$r$}{}}
For the surface integral for a sphere with radius $r$, we have for $\bm y=r \bm x$, $\bm x\in \S^d, \bm y \in \S^d_r$,
\begin{equation}\label{eq:int_sphere_r}
\int_{\S^d_r} h(\bm y) \d \mu_{d}(\bm y) = r^d \int_{\S^d} h(r\bm x) \d \mu_{d}(\bm x).
\end{equation}
We use for $h\colon \S^{d} \rightarrow \R$ the following decomposition of the integral on the sphere,
\begin{align}\label{eq:split_integral}
\int_{\S^d} h(\x) \d \mu_{d}(\x)
 &\stackrel{\eqref{eq:subs_Bsp}}{=}\int_{\B^{|\bm u|}} \changed{\left(1-\norm{\bm x_{\bm u}}^2\right)^{|\bm u^c|/2-1}}
   \int_{\S^{|\bm u^c|-1}} h\left(\x_{\bm u},\sqrt{1-\norm{\x_{\bm u}}^2}\, \x_{\bm u^c}\right)  \d \mu_{|\bm u^c|-1}(\x_{\bm u^c}) \d \x_{\bm u} \notag\\
    &\stackrel{\eqref{eq:int_sphere_r}}{=}\int_{\B^{|\bm u|}} \frac{1}{\sqrt{1-\norm{\x_{\bm u}}^2}}
   \int_{ \S^{|\bm u^c|-1}_{\sqrt{1-\norm{\x_{\bm u}}^2}}} h\left(\x_{\bm u}, \x_{\bm u^c}\right)  \d \mu_{|\bm u^c|-1}(\x_{\bm u^c}) \d \x_{\bm u} .
\end{align}
Similarly, the integral on a sphere with different radius,
\begin{align}\label{eq:split_integral_r}
\int_{\S_r^d} h(r\x) \d \mu_{d}(r\x)
 &\stackrel{\eqref{eq:int_sphere_r}}{=}r^d \int_{\S^d} h(r\x) \d \mu_{d}(\x)\notag\\
 &\stackrel{\eqref{eq:split_integral}}{=} r^d \int_{\B^{|\bm u|}}\frac{1}{\sqrt{1-\norm{\x_{\bm u}}^2}} \int_{\S^{|\bm u^c|-1}_{\sqrt{1-\norm{\x_{\bm u}}^2}} }h(r\x) \d\mu_{d-|\bm  u|}(\x_{\bm u^c}) \d \x_{\bm u}.
\end{align}
For balls with different radius, we  have
\begin{equation}\label{eq:int_ball}
\int_{\B_r^d} h(\y) \d \y = r^d \int_{\B^d} h(r\x) \d \x.
\end{equation}

\subsection{Spherical harmonics}  
A polynomial $p \colon \R^{d+1} \to \R$ is called \textbf{homogeneous of degree $\ell$} if 
\[
p(r\x) = r^{\ell} p(\x) \quad \text{for all } r > 0, \ \x \in \mathbb{R}^{d+1}.
\]
Let $\mathcal P_{\ell}^d$ denote the space of real homogeneous polynomials of degree $\ell$. A polynomial is called  \textbf{harmonic} if it satisfies the Laplace equation
\[
\Delta p = 0, \qquad \Delta = \sum_{i=1}^{d+1} \frac{\partial^2}{\partial x_i^2}.
\]
For $\ell =0,1,\ldots, $ we denote the linear spaces of real harmonic polynomials of degree $\ell$ by 
$$\mathcal H^d_\ell\coloneqq \{p\in \mathcal P_{\ell}^d\mid \Delta p =0\}.$$
The restriction of a homogeneous harmonic polynomial $p$ of degree $\ell$ to the sphere $\mathbb{S}^d$ is called a  \textbf{spherical harmonic of degree $\ell$}.
We denote the corresponding space by
\[
\mathcal{H}_\ell(\mathbb{S}^d)
\coloneqq \bigl\{\, p|_{\mathbb{S}^d} \; \colon \; p  \text{ homogeneous, harmonic, and of degree }  \ell \,\bigr\},
\]
see for example~\cite{DaXu13, DuXu14} for more details about these spaces.
The spaces $\mathcal{H}_\ell(\mathbb{S}^d)$ form an orthogonal decomposition
\[
L_2(\mathbb{S}^d)
= \bigoplus_{\ell=0}^\infty \mathcal{H}_\ell(\mathbb{S}^d),
\]
and any $f \in L_2(\mathbb{S}^d)$ admits the spherical harmonic expansion
\[
f(\x) = \sum_{\ell=0}^\infty \sum_{m=1}^{\dim (\mathcal{H}_\ell(\mathbb{S}^d))} \hat f_{\ell,m} \, Y_{\ell,m}(\x),
\]
where $\{Y_{\ell,m}\}_{m=1}^{\dim (\mathcal{H}_\ell(\mathbb{S}^d))}$ is an orthonormal basis of $\mathcal{H}_\ell(\mathbb{S}^d)$ and $\dim (\mathcal{H}_\ell(\mathbb{S}^d))$ denotes its dimension.
\begin{remark}[Polynomial representatives and coordinate dependence]
Each spherical harmonic $Y_{\ell,m}\in\mathcal{H}_\ell(\mathbb{S}^d)$ is the restriction of a 
homogeneous harmonic polynomial $P_{\ell,m}:\mathbb{R}^{d+1}\to\mathbb{R}$ of degree $\ell$, i.e.
\[
Y_{\ell,m}(\x) = P_{\ell,m}(\x)\big|_{\mathbb{S}^d}.
\]
The polynomial $P_{\ell,m}$ can be expressed as a linear combination of monomials
\[
x_1^{\alpha_1}\,x_2^{\alpha_2}\cdots x_{d+1}^{\alpha_{d+1}}, 
\qquad |\bm\alpha| := \alpha_1+\cdots+\alpha_{d+1} = \ell,
\]
whose coefficients are determined by the harmonicity condition $\Delta P_{\ell,m}=0$.  
Thus, every spherical harmonic of degree $\ell$ involves monomials of total degree $\ell$ in the Cartesian coordinates.  
In particular, the index $\ell$ corresponds to the total polynomial degree, while the secondary index $m$ enumerates the independent harmonic combinations of these monomials, forming a basis of the space
\[
\mathcal{H}_\ell(\mathbb{S}^d)
=\operatorname{span}\bigl\{\,P_{\ell,m}(\x)\big|_{\mathbb{S}^d}: 1\le m\le \dim (\mathcal{H}_\ell(\mathbb{S}^d))\,\bigr\}.
\]
Consequently, each $Y_{\ell,m}$ generally depends on all coordinates $x_1,\dots,x_{d+1}$, though with specific parity and symmetry relations determined by the monomial exponents $\bm\alpha$.
\end{remark}
For $d=2$, $\dim (\mathcal{H}_\ell(\mathbb{S}^d))=2$. In polar coordinates $(x_1,x_2)=(r\cos\theta,r\sin\theta)$ of $\R^2$, this basis is given by
\begin{equation}\label{eq:Y_for_d=2}
    Y_{\ell}^{(1)}=r^\ell \cos(\ell \theta), \quad  Y_{\ell}^{(2)}=r^\ell \sin(\ell \theta).
\end{equation}

\subsection{Decomposition into odd and even parts}\label{sec:even_odd}
Let the domain be $\mathbb D\in \{\S^d,\B^{d+1}\} $. We analyze the parity properties of the underlying functions. In the one-dimensional case, a function $f$ is defined as \textbf{odd} if $f(-x)=-f(x)$ and as \textbf{even} if $f(-x)=f(x)$ almost everywhere. Then, an arbitrary function $f$ is decomposed into its even and odd parts by
\[
f(x) = f_{\mathrm e}(x) + f_{\mathrm o}(x), \qquad
f_{\mathrm e}(x) = \tfrac12(f(x)+f(-x)), \quad
f_{\mathrm o}(x) = \tfrac12(f(x)-f(-x)).
\]
Similarly, we do this for a function on $\S^d$ or $\B^{d+1}$ by using the parity vector $\bm \xi=(\xi_i)_{i=1}^{d+1}\in \{0,1\}^{d+1}$ and defining
\begin{align}\label{eq:def_even_odd_d}
f(\x) &= \sum_{\bm \xi\in \{0,1\}^{d+1}} [\Xi_{\bm\xi} f](\x), \\
\Xi_{\bm \xi}\colon L_2(\mathbb D)\rightarrow L_2(\mathbb D)\quad \text{ with }\quad [\Xi_{\bm \xi} f](\x) &\coloneqq \frac{1}{2^{d+1}} \sum_{\bm k \in \{-1,1\}^{d+1}} \left(\underbrace{\left(\prod_{i=1}^{d+1}k_i^{\xi_i}\right)}_{\in \{-1,1\}} f(\bm k \odot \x )\right)\notag
\end{align}
where $\xi_i=1$ belongs to an odd function in direction $i$ and $\xi_i=0$ belongs to an even function in direction $i$. In~\eqref{eq:def_even_odd_d} the construction of the terms is exactly such that one factor $(-1)$ is added for every $i\in [d+1]$, for which $k_i=-1$ and $\xi_i=1$.
\begin{definition}
We denote the vector $\bm \xi$ as the \textbf{parity} of the function $g$, if for all $i\in [d+1]$, for almost all $\x\in \mathbb D$,
\begin{align*}
\begin{split}
\xi_i =0 &\Leftrightarrow g(\x)  = g(\x_{[d+1]\backslash i}, -x_i), \\
\xi_i =1 &\Leftrightarrow g(\x)  = -g(\x_{[d+1]\backslash i}, -x_i)  .
\end{split}
\end{align*}
For $f\in L_2(\mathbb D)$ we define the corresponding function spaces 
\begin{equation*}
L_2(\bm \xi) = \{f\in L_2(\mathbb D)\mid f \text{ has parity } \bm \xi\}.
 \end{equation*}
The spaces $L_2(\bm \xi)$ are orthogonal, i.e.
$$L_2(\mathbb D)= \bigoplus_{\bm \xi\in \{0,1\}^{d+1}} L_2(\bm \xi).$$
\end{definition}

\section{Calculation of the spherical ANOVA decompositions in Section~\ref{sec:numerics}}\label{sec:app1}
In this section we calculate the spherical ANOVA decomposition from Definition~\ref{def:spherical_anova_final} for some of the test functions. We use in the following these integrals frequently,
\begin{align*}
\int_{\S^{d-1}}x_i^2 \d \mu_{d-1}(\x) &= \frac{\omega_{d-1}}{d}\quad \text{ if } i\in [d+1] \\
\int_{\S^{d-1}}x_i^{2k+1} \d \mu_{d-1}(\x) &= 0 \quad \text{ if } k\in \N, i\in [d+1].
\end{align*}
\begin{itemize}

\item$f_A(\x)=x_1  x_2^3 + 2 x_3  x_4^5 + 0.05x_5   $\\
The parity decomposition~\eqref{eq:def_even_odd_d} of this function is
\begin{equation*}
\Xi_{\bm \xi} f_A(\x)= \begin{cases}
 x_1x_2^3 & \text{ if } \bm \xi = \bm e_1 +\bm e_2  \\
  2 x_3  x_4^5 & \text{ if } \bm \xi = \bm e_3 +\bm e_4  \\
   0.05x_5 & \text{ if } \bm \xi = \bm e_5   \\
 0& \text{ otherwise }.
\end{cases}
\end{equation*}

Applying the operators $P_{\bm u}$ to this function yields that 
\begin{align*}
P_{\{1,2\}}f_A(x_1,x_2) &=  x_1x_2^3\\
P_{\{3,4\}} f_A(x_3,x_4) &=  2 x_3  x_4^5  \\
P_{\{5\}} f_A(x_5) &= 0.05x_5 
\end{align*}
Since all the terms are odd in all directions, the spherical ANOVA decomposition from Definition~\ref{def:spherical_anova_final} directly gives that
\begin{align*}
f_{\varnothing,\bm 0 } &= 0,\\
f_{\{1,2\},\bm e_1 + \bm e_2 }(x_1,x_2) &= x_1x_2^3\\
f_{\{3,4\},\bm e_3 + \bm e_4}(x_3,x_4) &=2 x_3  x_4^5\\
f_{\{5\},\bm e_5 } (x_5)&=0.05x_5,
\end{align*}
and all other terms are zero.

\item$f_B(\x)= x_2x_1^2   $\\
The parity decomposition~\eqref{eq:def_even_odd_d} of this function is
\begin{equation*}
\Xi_{\bm \xi}f_B(\x) = \begin{cases}
 x_2x_1^2 & \text{ if } \bm \xi = \bm e_2 \\
 0& \text{ otherwise }.
\end{cases}
\end{equation*}
Applying the operators $P_{\bm u}$ with $2\in \bm u$ to this function yields that 
\begin{small}
\begin{align*}
P_{2}f_B(x_2) &= \frac{1}{d}x_2(1-x_2^2)\\
P_{\{1,2\}} f_B(x_1,x_2) &=x_2x_1^2  \\
P_{\{2\}\cup \bm a} f_B(x_{\{2\}\cup \bm a}) &= \frac{1}{\omega_{d-1-|\bm a|}}x_2\int_{\S^{d-|\bm a|-1}} (1-\norm{\bm x_{\bm a}}^2-x_2^2)x_1^2 \d \mu_{d-|\bm a|-1}(\x_{\bm a^c\backslash 2})\\
&=\frac{x_2}{d-|\bm a|} (1-\norm{\x_{\bm a}}^2-x_2^2) 
&\quad \text{if }\bm a\subset [d+1]\backslash \{1,2\}  \\
P_{\{1,2\}\cup \bm a} f_B& = f_B &\quad \text{if }\bm a\subset [d+1]\backslash \{1,2\}
\end{align*}
\end{small}

Here we see that the spherical ANOVA decomposition from Definition~\ref{def:spherical_anova_final} blurs the original function (which depends only on the variables $x_1$ and $x_2$ also to the other variables due to the geometric properties on the sphere. But with \changed{Definition~\ref{def:final_def}} we describe the spherical ANOVA decomposition by
\begin{small}
\begin{align*}
f_{\varnothing,\bm 0 } &= 0,\\
f_{2,\bm e_2 }(x_2) &= \frac{1}{d}x_2-\frac{1}{d}x_2^3\\
f_{\{1,2\},\bm e_2}(x_1,x_2) &=x_2x_1^2 - \frac{1}{d}x_2+\frac{1}{d}x_2^3.
\end{align*}
\end{small}
All other terms are zero. 

\item $f_C(\x)= x_1^4 + x_2^2 $\\
This is a pure even function, such that the parity decomposition of this function is 
\begin{equation*}
\Xi_{\bm \xi}f_C(\bm x) = \begin{cases}
x_1^4 +x_2^2 & \text{ if } \bm \xi = \bm 0, \\
 0& \text{ otherwise.}
\end{cases}
\end{equation*}
Applying the operators $P_{\bm u}$ to this function yields due to the Beta function,
\begin{allowdisplaybreaks}
\begin{align*}
P_{\varnothing}x_1^4 & = \frac{1}{\omega_d}\int_{\S^d} f_C \d \mu_{d}(\x)
=\frac{\omega_{d-1}}{\omega_d} \int_{-1}^1 x_1^4 (1-x_1^2)^{\tfrac{d-2}{2}}\d x_1\\\
&=\frac{\Gamma(\frac{d+1}{2})}{\sqrt{\pi}\Gamma(\frac{d}{2})} 2 \, \int_{0}^1 t^2 (1-t)^{\tfrac{d-2}{2}}\frac{1}{2\sqrt{t}}\d t
=\frac{\Gamma(\frac{d+1}{2})}{\sqrt{\pi}\Gamma(\frac{d}{2})} \int_{0}^1 t^{3/2} (1-t)^{\tfrac{d-2}{2}}\d t\\
&=\frac{\Gamma(\frac{d+1}{2})}{\sqrt{\pi}\Gamma(\frac{d}{2})}\, \frac{\Gamma(\tfrac 52) \Gamma(\tfrac d2)}{\Gamma(\frac{d+5}{2})}
=\frac{\Gamma(\frac{d+1}{2})}{\sqrt{\pi}}\, \frac{3\sqrt{\pi} }{4\,\Gamma(\frac{d+5}{2})}
=\frac{3\Gamma(\frac{d+1}{2}) }{4\,\Gamma(\frac{d+5}{2})}
=\frac{3 }{4\, \tfrac{d+3}{2}\tfrac{d+1}{2}}=\frac{3 }{(d+3)(d+1)},\\
P_{\varnothing}x_2^2 &  
=\frac{\omega_{d-1}}{\omega_d} \int_{-1}^1 x_2^2 (1-x_2^2)^{\tfrac{d-2}{2}}\d x_2\\\
&=\frac{\Gamma(\frac{d+1}{2})}{\sqrt{\pi}\Gamma(\frac{d}{2})} 2 \, \int_{0}^1 t (1-t)^{\tfrac{d-2}{2}}\frac{1}{2\sqrt{t}}\d t
=\frac{\Gamma(\frac{d+1}{2})}{\sqrt{\pi}\Gamma(\frac{d}{2})} \int_{0}^1 t^{1/2} (1-t)^{\tfrac{d-2}{2}}\d t\\
&=\frac{\Gamma(\frac{d+1}{2})}{\sqrt{\pi}\Gamma(\frac{d}{2})}\, \frac{\Gamma(\tfrac 32) \Gamma(\tfrac d2)}{\Gamma(\frac{d+3}{2})}
=\frac{\Gamma(\frac{d+1}{2})}{\sqrt{\pi}}\, \frac{\tfrac 12\sqrt{\pi} }{\Gamma(\frac{d+3}{2})}
=\frac{\Gamma(\frac{d+1}{2}) }{2\,\Gamma(\frac{d+3}{2})}
=\frac{1 }{d+1},\\
P_{\{1\}}f_C &  = x_1^4 + \frac{1}{d} (1-x_1^2) , \\
P_{\{2\}}f_C(x_2) &= \tfrac{3}{d(d+2)} (1-x_2^2)^2  + x_2^2,\\
P_{\{i\}}f_C(x_i) &= \tfrac{3}{d(d+2)} (1-x_i^2)^2 + \frac{1}{d} (1-x_i^2) \quad i\in \{3,4,5,\ldots,d+1\}
\\
P_{\{1,2\}}f_C(x_1,x_2) &= x_1^4 + x_2^2.
\end{align*}
\end{allowdisplaybreaks}
The spherical ANOVA decomposition \changed{from Definition~\ref{def:final_def}} of this function, with regard to the cancellations in Section~\ref{sec:challenge} then is
\begin{align*}
f_{\varnothing,\bm 0 } & = \frac{d+9}{(d+3)(d+1)},\\
f_{1,\bm e_1 }(x_1) &= 0, &f_{1,\bm 0 }(x_1) &= x_1^4 -  \frac{3 }{(d+3)(d+1)},\\
f_{2,\bm e_2 }(x_2) &= 0, &f_{2,\bm 0 }(x_2) &= x_2^2  -\frac{1}{d+1},
\end{align*}
and all other terms are zero. 
\end{itemize}
\end{appendix}

\end{document}